\documentclass{compositio}
\usepackage[utf8]{inputenc}

\usepackage[usenames,dvipsnames]{color}

\usepackage{amsthm,amsfonts,amssymb,amsmath}
\usepackage[all]{xy}
\usepackage{xr-hyper}
\usepackage[colorlinks=
   citecolor=Black,
   linkcolor=Red,
   urlcolor=Blue]{hyperref}
\usepackage{verbatim}
\usepackage{pdfsync}

\newcommand{\BC}{{\mathbb {C}}}

\newcommand{\BF}{{\mathbb {F}}}

\newcommand{\BL}{{\mathbb {L}}}

\newcommand{\BN}{{\mathbb {N}}}

\newcommand{\BP}{{\mathbb {P}}}
\newcommand{\BQ}{{\mathbb {Q}}}

\newcommand{\BV}{{\mathbb {V}}}

\newcommand{\BX}{{\mathbb {X}}}
\newcommand{\BY}{{\mathbb {Y}}}
\newcommand{\BZ}{{\mathbb {Z}}}

\newcommand{\CF}{{\mathcal {F}}}

\newcommand{\CL}{{\mathcal {L}}}
\newcommand{\CM}{{\mathcal {M}}}
\newcommand{\CN}{{\mathcal {N}}}
\newcommand{\CO}{{\mathcal {O}}}

\newcommand{\CT}{{\mathcal {T}}}

\newcommand{\CV}{{\mathcal {V}}}
\newcommand{\CW}{{\mathcal {W}}}

\newcommand{\CY}{{\mathcal {Y}}}
\newcommand{\CZ}{{\mathcal {Z}}}

\newcommand{\Hom}{{\mathrm{Hom}}}

\newcommand{\ord}{{\mathrm{ord}}}

\newcommand{\Spec}{{\mathrm{Spec}}}

\newtheorem{thm}{Theorem}[section]

\newtheorem{prop}[thm]{Proposition}

\newtheorem{theorem}{Theorem}[section]
\newtheorem{proposition}[theorem]{Proposition}
\newtheorem{lemma}[theorem]{Lemma}
\newtheorem {conjecture}[theorem]{Conjecture}
\newtheorem{corollary}[theorem]{Corollary}

\theoremstyle{definition}
\newtheorem{definition}[theorem]{Definition}

\newtheorem{remark}[theorem]{Remark}

\numberwithin{equation}{section}

\begin{document}

\title{On the Arithmetic Fundamental Lemma in the minuscule case}

\author{Michael Rapoport}
\address{Mathematisches Institut der Universit\"at Bonn\\ 
Endenicher Allee 60\\
53115 Bonn, Germany\\
email: rapoport@math.uni-bonn.de}

\author{Ulrich Terstiege} 
\address{Institut f\"ur Experimentelle Mathematik\\
Universit\"at Duisburg-Essen, Campus Essen\\
Ellernstra{\ss}e 29\\
45326 Essen, Germany\\
email: ulrich.terstiege@uni-due.de}

\author{Wei Zhang}
\address{Department of Mathematics\\
Columbia University\\
New York, NY 10027, USA \\
email: wzhang@math.columbia.edu 
}

\thanks{Research of Rapoport and Terstiege partially supported by SFB/TR 45 ``Periods, Moduli Spaces and 
Arithmetic of Algebraic Varieties" of the DFG. Research of Zhang partially supported by NSF grant DMS 1204365.}


\begin{abstract}
{The arithmetic fundamental lemma conjecture of the third author connects the derivative of an orbital integral on a symmetric space with an intersection number on a formal moduli space of $p$-divisible groups of Picard type. It arises in the relative trace formula approach to the arithmetic Gan-Gross-Prasad conjecture. We prove this conjecture in the minuscule case.}
\end{abstract}

\date{\today}
\maketitle

\tableofcontents

\section{Introduction} 
In this introduction, we first formulate (a variant of)  the fundamental lemma conjecture (FL) of Jacquet-Rallis \cite{JR}  and the arithmetic fundamental lemma conjecture (AFL) of the third author \cite{Z}. Then we state our main result, which is a confirmation of the second  conjecture in arbitrary dimension under restrictive conditions.   

Let $p$ be an odd prime. Let $F$ be a finite extension of $\BQ_p$, with ring of integers $\CO_F$, uniformizer $\pi$  and residue field $k$ with $q$ elements. Let $E$ be an unramified  quadratic extension, with ring of integers $\CO_E$, and residue field $k'$. We denote the non-trivial element in ${\rm Gal}(E/F)$ by $\sigma$ or by $a\mapsto \bar a$. Also, we denote by $\eta=\eta_{E/F}$ the quadratic character of $F^\times$ corresponding to $E/F$.

Let $n\geq 1$. Let $v=(0, 0,\ldots,0,1)\in F^n$. We denote by $F^{n-1}$ the subspace of vectors in $F^n$ with trivial last entry. We have a canonical inclusion $GL_{n-1}\hookrightarrow GL_n$ of algebraic groups over $F$. An element $g\in GL_n(E)$ is called regular semi-simple (with respect to the action of $GL_{n-1}(E)$ by conjugation) if both the vectors $(g^iv)_{i=0,\ldots,n-1}$ and the vectors $(^t\!vg^i)_{i=0,\ldots,n-1}$ are linearly independent. This property is equivalent to the condition that the stabilizer ${\rm Stab}_{GL_{n-1}}(g)$ is trivial, and that the orbit  of $g$ under $GL_{n-1}$ is Zariski closed in $GL_n$, cf.\ \cite[Theorem 6.1]{RS}\footnote{In \cite{RS}, the Lie algebra version is considered. But it is easy to deduce the group version from the Lie algebra version. Moreover, what is called  ``regular semi-simple" here is called ``regular" in \cite{RS}.}.  To $g\in GL_n(E)$ we associate the following numerical invariants: the coefficients of the characteristic polynomial 
 ${\rm char}_g(T)\in E[T]$, and the $n-1$ elements $^t\!vg^iv\in E$ for $ i=1,\ldots,n-1$.  Then two regular semi-simple elements are conjugate under an element of $GL_{n-1}(E)$ if and only if they have the same invariants, cf. \cite{Z}. 

Let 
\begin{equation}
S_n(F)=\{ s\in GL_n(E)\mid s\sigma(s)=1\} .
\end{equation}
Then $GL_{n-1}(F)$ acts on $S_n(F)$, and two elements in $S_n(F)$ which are regular semi-simple (as elements of $GL_n(E)$) are conjugate under $GL_{n-1}(E)$ if and only if they are conjugate under $GL_{n-1}(F)$. 

Let $J\in {\rm Herm}_{n-1}(E/F)$ be a hermitian matrix of size $n-1$. It defines a hermitian form on $E^{n-1}$. We obtain a hermitian form $J\oplus 1$ of size $n$, which corresponds to extending the hermitian form to $E^n$ by adding an orthogonal vector $u$ of length $1$. We obtain an inclusion of unitary groups
$$
U(J)(F)\hookrightarrow U(J\oplus 1)(F) ,
$$
and therefore an action of $U(J)(F)$ on $U(J\oplus 1)(F)$ by conjugation. We consider $U(J\oplus 1)(F)$ as a subset of $GL_n(E)$ in the obvious way by sending $u$ to $v$, and $E^{n-1}$ to the subspace of vectors with trivial last entry. We call an element $g\in U(J\oplus 1)(F)$ regular semi-simple if it is regular semi-simple as an element of $GL_n(E)$. Two regular semi-simple elements $\gamma\in S_n(F)$ and $g\in U(J\oplus 1)(F)$ are said to match if they are conjugate under $GL_{n-1}(E)$ (when considered as elements of $GL_n(E)$), or, equivalently, if they have the same invariants. This property only depends on the orbits of $\gamma$ under $GL_{n-1}(F)$, resp. $g$ under $U(J)(F)$. This matching condition defines a bijection between orbit spaces \cite{Z}, Lemma~2.3, 
\begin{equation}
\big[U(J_0\oplus 1)(F)_{\rm rs}\big] \sqcup \big[U(J_1\oplus 1)(F)_{\rm rs}\big]\simeq \big[S_n(F)_{\rm rs}\big] .
\end{equation}
Here $J_0$ denotes the split hermitian form, and $J_1$ the non-split hermitian form, i.e., the discriminant of $J_0$ has even valuation, and the discriminant of $J_1$ has odd valuation. 

For $\gamma\in S_n(F)_{\rm rs}$, and $f\in C^\infty_c(S_n(F))$, consider the weighted orbital integral
\begin{equation}
O(\gamma, f)=\int_{GL_{n-1}(F)} f(h^{-1}\gamma h)\eta({\rm det}\,h)dh ,
\end{equation}
where we normalize the measure so that $GL_{n-1}(\CO_F)$ has measure $1$. Similarly, for any $g\in U(J_0\oplus 1)(F)_{\rm rs}$, and $f\in C^\infty_c(U(J_0\oplus 1)(F))$, we form  the  orbital integral
\begin{equation}
O(g, f)=\int_{U(J_0)(F)} f(h^{-1}g h)dh ,
\end{equation}
where we normalize the measure so that the stabilizer $K'$ of a self-dual lattice $\Lambda'$ in $E^{n-1}$ has measure $1$. Let $K$ be the stabilizer of the self-dual lattice $\Lambda=\Lambda'\oplus \CO_Eu$. 

The FL is now the following statement (for the ``Lie algebra'' version see \cite{JR}). 
\begin{conjecture}
For $\gamma\in S_n(F)_{\rm rs}$,
\begin{equation*}
\begin{aligned} O(\gamma, 1_{S_n(\CO_F)})=\begin{cases} {\omega(\gamma)O(g, 1_{K})  \text{\rm \, if $\gamma$ matches 
$g\in U(J_0\oplus 1)(F)_{\rm rs}$ }} ,\\
\\
{0 \quad\quad\quad\quad\quad\text{\rm \, if $\gamma$ matches no $g\in U(J_0\oplus 1)(F)_{\rm rs}$ . }}
\end{cases}
\end{aligned}
\end{equation*}
\end{conjecture}
Here the sign $\omega(\gamma)$ is given by 
\begin{equation}
\omega(\gamma)=(-1)^{v({\rm det}(\gamma^iv)_{i=0,\ldots, n-1})} .
\end{equation}
Both orbital integrals appearing in the conjecture count certain $\CO_E$-lattices in $E^n$. Let $L=L_g$ be the lattice generated by the vectors $u, gu,\ldots g^{n-1}u$, where we recall the vector $u$ of length one from above. Then the first clause of the above identity can be written as  
$$
\omega(\gamma)\sum_{\{\Lambda\mid L\subset \Lambda\subset L^*, g\Lambda=\Lambda,\Lambda^\tau=\Lambda\}} (-1)^{\ell(\Lambda/L)} =\omega(\gamma)\sum_{\{\Lambda\mid L\subset \Lambda\subset L^*, g\Lambda=\Lambda, \Lambda^*=\Lambda\}}1.
$$ Here $\tau$ is the antilinear involution on $E^n$, depending on $g$, which sends $g^iu$ to $g^{-i}u$ for  $i=0,\ldots, n-1$. Also, for any lattice $\Lambda$, we denote by $\Lambda^*$ the lattice of elements of $E^n$ which pair integrally with all elements of $\Lambda$ ({\it dual  lattice}). 

The equal characteristic analogue of FL was proved by Z. Yun, for $p>n$; J. Gordon deduced  FL in the $p$-adic case, for $p$ large enough (but unspecified), cf.  \cite{GY}. 

Now we come to the AFL conjecture. For $\gamma\in S_n(F)_{\rm rs}$, and $f\in C^\infty_c(S_n(F))$, and $s\in \BC$, let 
$$
O(\gamma, f, s)=\int_{GL_{n-1}(F)} f(h^{-1}\gamma h)\eta({\rm det}\, h)|{\rm det} h|^s dh ,
$$
and introduce
\begin{equation}
O'(\gamma, 1_{S_n(\CO_F)})= \frac{d}{ds}O(\gamma, 1_{S_n(\CO_F)},s)_{\big| s=0} . 
\end{equation}
Then the conjecture is as follows.
\begin{conjecture}\label{AFL}
 For $\gamma\in S_n(F)_{\rm rs}$ which matches $g\in U(J_1\oplus 1)(F)_{\rm rs}$, 
$$O'(\gamma, 1_{S_n(\CO_F)})=-\omega(\gamma)\big\langle\Delta(\CN_{n-1}), ({\rm id}\times g)\Delta (\CN_{n-1}) \big\rangle .$$ 
\end{conjecture}
On the RHS appears the arithmetic intersection product of two formal subschemes inside the formal scheme $\CN_{n-1}\times_{{\rm Spf}\, \CO_{\breve F}}\CN_n$. Here  $\CN_n$ denotes the moduli space over the ring of integers in the completion of the maximal unramified extension $\breve F$ of $F$ of formal $\CO_F$-modules of height $n$ with 
$\CO_E$-action of signature $(1, n-1)$ and with principal polarization compatible with the involution $\sigma$ on $\CO_E$. ($\CN_n$ is a special case of an {\it RZ-space} \cite{RZ}.) Similarly for $\CN_{n-1}$,
which is naturally embedded in $\CN_n$. The element  $g\in U(J_1\oplus 1)(F)$ acts on $\CN_n$ in a natural way. Then $\Delta(\CN_{n-1})$ and $ ({\rm id}\times g)\Delta (\CN_{n-1}) $ are two formal schemes of formal dimension $n-1$, contained in the formal scheme $\CN_{n-1}\times_{{\rm Spf}\, \CO_{\breve F}}\CN_n$ of formal dimension $2(n-1)$, i.e., we are in a situation of middle dimension intersection.
We refer to \cite{Z} for the precise definition of $\CN_n$, and for the definition of the intersection product, and the proof of the fact that the RHS is a finite quantity (cf. also \S\S 2--4 below). 

In the following, we fix $n\geq 2$ and denote $\CN_n$ simply by $\CN$, and $\CN_{n-1}$ by $\CM$. 

As before, the LHS can be expressed in a  combinatorial way, as
$$ 
\omega(\gamma){\rm log}\, q\sum_{\{\Lambda\mid L\subset \Lambda\subset L^*, g\Lambda=\Lambda,\Lambda^\tau=\Lambda\}} (-1)^{\ell(\Lambda/L)} \ell(\Lambda/L) . 
$$ In the case that the intersection of the formal schemes $\Delta(\CM)$ and  $({\rm id}\times g)\Delta (\CM)$ is proper, i.e., is a set of isolated points, and using the Bruhat-Tits stratification of $\CN_{\rm red}$ \cite{VW}, the RHS can also be written as a sum over lattices, as
$$
-\omega(\gamma){\rm log}\, q \sum_{\{\Lambda\mid L\subset \Lambda\subset L^*, g\Lambda=\Lambda, \pi\Lambda\subset\Lambda^*\subset\Lambda\}}{\rm mult}(\Lambda) .
$$
Here the number ${\rm mult}(\Lambda)$ is the intersection multiplicity of $\Delta(\CM)$ and $({\rm id}_\CM\times g)\Delta(\CM)$ {\it along the stratum 
$\CV(\Lambda)^o$}. 

We note that this conjecture holds true for $n=2, n=3$, by results of the third author \cite{Z}. In these cases the intersection appearing above is automatically proper. 

We now come to the description of the results of this paper, which are valid for any $n$ but with strong restrictions on $g$. Let 
$$(\BZ^n)_+=\{(r_1,\ldots,r_n)\in \BZ^n\mid r_1\geq\ldots\geq r_n\}.
$$  Let ${\rm inv}(g)=(r_1,\ldots,r_n)\in (\BZ^n)_+$ be the unique element such that $L_g^*$ has a basis $e_1,\ldots, e_n$ such that $\pi^{r_1}e_1,\ldots, \pi^{r_n}e_n$ is a basis of $L_g$. Note that $r_n=0$, and that $\sum_i r_i$ is odd. It turns out that the `bigger' ${\rm inv}(g)$ is, the more difficult it is to prove the identity in AFL. From this point of view we treat here the simplest non-trivial case. 
\begin{theorem}\label{mainthm} Let $g\in U(J_1\oplus 1)(F)_{\rm rs}$. 

\smallskip

\noindent (i) The underlying reduced scheme  of the intersection $\Delta(\CM)\cap ({\rm id}_\CM\times g)\Delta(\CM)$ has a stratification by Deligne-Lusztig varieties (we refer to \S\ref{fpinstratum} for the precise description of which Deligne-Lusztig varieties can occur). 

\smallskip

\noindent {\rm From now on assume that ${\rm inv}(g)$ is minuscule, i.e., that ${\rm inv}(g)=(1^{(m)}, 0^{(n-m)})$,
for some $m\geq 1$. Then:}

\smallskip

\noindent (ii) The intersection of  $\Delta(\CM)$ and $({\rm id}_\CM\times g)\Delta(\CM)$ is proper. Furthermore, the arithmetic intersection product 
$\big\langle\Delta(\CM), ({\rm id}\times g)\Delta (\CM) \big\rangle$ is equal to 
$$
{\rm log}\, q \sum_{x\in (\Delta(\CM)\cap({\rm id}\times g)\Delta (\CM))(\bar k)} \ell\big({\CO_{\Delta(\CM)\cap ({\rm id}\times g)\Delta(\CM), x}}\big) ,
$$
i.e., there are no higher Tor-terms. 

\smallskip

\noindent (iii) The intersection of  $\Delta(\CM)$ and $({\rm id}_\CM\times g)\Delta(\CM)$ is concentrated in the special fiber, i.e., the uniformizer $\pi$ annihilates its structure sheaf. 

\smallskip

\noindent (iv) The AFL identity holds, provided that $n\leq 2p-2$. Furthermore, in this case  the lengths of the local rings appearing in (ii) are all identical. 

\end{theorem}

Assertion (i) is proved in section 6, and (ii) follows from Propositions \ref{scheme} and \ref{fixpointstr}. Assertion (iii) follows from Theorem \ref{speciald}, and assertion (iv) follows from Propositions \ref{cardform}, \ref{derivform},  and \ref{length}.  That the lengths of all local rings are identical follows from our explicit determination of these lengths, although we think that there should be an {\it a priori} proof, without the restriction on $p$.

In fact, we will prove the assertions above only in the case when $F=\BQ_p$, because in this case we can refer to \cite{V} and \cite{VW} for the structure of $\CN_n$, and also to \cite{KR2} for some global results. However, there is no doubt that the results should generalize to arbitrary $p$-adic fields. 

There is a fundamental difference between  the seemingly very similar combinatorial descriptions of both sides in the FL and in the AFL. Whereas in the FL there is a rather simple criterion to decide when  both sides of the identity are non-zero, the corresponding question for the AFL seems very subtle in general. However, in the case of a minuscule element $g$, we give  a simple criterion in terms of the induced automorphism of the $k'$-vector space $L_g^*/L_g$ to decide when the two sides of the AFL identity are non-zero, cf. \S\ref{theminusculecase}. 

There is some relation between the AFL problem and the problem of intersecting {\it special divisors} considered in \cite{KR}. Indeed, the intersection $\Delta(\CM)\cap({\rm id}_\CM\times g)\Delta(\CM)$ is contained in the intersection of the special divisors (in the sense of \cite{KR}) $\CZ(g^iu)$, for $i=0,\ldots, n-1$. Then point (iii) of Theorem \ref{mainthm} is a consequence of the following theorem, which is of independent interest.
\begin{theorem}\label{secondmain} 
Consider the intersection of special divisors  $\CZ(x_1),\ldots\CZ(x_n)$ on $\CN_n$, where the fundamental matrix (in the sense of \cite{KR}) is equivalent to the diagonal matrix ${\rm diag}(\pi^{(m)}, 1^{(n-m)})$. Then this intersection is concentrated in the special fiber, and is in fact equal to a  closed Bruhat-Tits stratum of type $m$ of $(\CN_{n})_ {\rm red}$. 
\end{theorem}
Again, we prove this only in the case  $F=\BQ_p$. We view this theorem as a confirmation of the following conjecture in a special case.
\begin{conjecture}Consider the intersection of special divisors $\CZ(x_1),\ldots,\CZ(x_n)$ on $\CN_n$, where the fundamental matrix is equivalent to the diagonal matrix ${\rm diag}(\pi^{r_1}, \pi^{r_2}, \ldots,\pi^{r_n})$ with $r_1\geq r_2\geq\ldots\geq r_n$. Then
$\pi^{r_1}$ annihilates the structure sheaf of $\CZ({x_1})\cap\CZ({x_2})\ldots\cap\CZ({x_n})$.

\end{conjecture}

The lay-out of the paper is as follows. In sections 2 and 3 we recall some facts about the formal moduli spaces $\CN_n$ and the geometry of their underlying schemes. In section 4 we explain the intersection product appearing on  the RHS of Conjecture \ref{AFL}. In sections 5 and 6 we address the problem of determining the underlying point set of the intersection. More precisely, we write in section 5 this intersection  $\Delta(\CM)\cap ({\rm id}_\CM\times g)\Delta(\CM)$ as a disjoint union over Bruhat-Tits strata of certain fixed point sets in each stratum. The determination of the individual fixed point sets then becomes a problem in Deligne-Lusztig theory that is discussed in section 6. In particular, we give a criterion for when this fixed point set is finite. In section 7  we explain the statements of the FL and the AFL, and show that  these conjectures can be interpreted as elementary counting expressions of lattices, as mentioned above. In the rest of the paper we concentrate on the minuscule case. In section 8 we determine the cardinality of the intersection $\Delta(\CM)\cap ({\rm id}_\CM\times g)\Delta(\CM)$ and calculate the LHS of Conjecture \ref{AFL}, which turn out to be  amusing combinatorial exercises. In section 9 we reduce the calculation of the length of the local ring at each point of this intersection to Theorem \ref{speciald}, alias Theorem \ref{secondmain} above,   and Theorem \ref{idealk}. These theorems are then proved in sections 10 and 11. Here the main tool is Zink's theory of displays of formal groups.

We conclude this introduction with a few remarks and questions. One remark is that we find it striking that the intersection $\Delta(\CM)\cap ({\rm id}_\CM\times g)\Delta(\CM)$ may be a discrete point set, but not consist entirely of {\it superspecial} points. This is in contrast to what occurs, e.g., in \cite{KR}, or \cite{Z}. A question  that seems very interesting to us is to clarify the relationship between the regular semi-simplicity of an element $g\in U(J_1\oplus 1)(F)$ and the finiteness of the intersection $\Delta(\CM)\cap ({\rm id}_\CM\times g)\Delta(\CM)$: it is easy to see that there are regular semi-simple elements $g$ such that the length of this  intersection is not finite. It would be very interesting to characterize those regular semi-simple elements with corresponding proper intersection, in analogy with the corresponding characterization in \cite{KR} of the cases when the intersection of special divisors is finite. For instance, if $g$ is not regular semi-simple, is the intersection $\Delta(\CM)\cap ({\rm id}_\CM\times g)\Delta(\CM)$  of infinite length?

We thank J.-B. Bost, B. Gross, X. He, S. Kudla, G. Lusztig, P. Scholze, J.-L. Waldspurger, X.~Zhu, and Th. Zink for helpful discussions. We also thank the referee for his/her careful reading of the paper. 

Parts of this work were done during research stays of the first author at Harvard University, of the first two authors at the Erwin Schr\"odinger Institute in Vienna, and of the third author at the Morningside Center of Mathematics and MSC of Tsinghua University at Beijing. We thank these institutions for their hospitality. 
\bigskip

\noindent{\bf Notation} Throughout the paper, we make the blanket assumption that $p$ is odd. We fix an algebraic closure $\bar k=\BF$ of $k$, and denote by $\sigma$ the relative Frobenius in ${\rm Gal}(\BF/k)$. We identify $k'$ with the quadratic extension of $k$ in $\BF$.

\section{The set up}\label{setup}
Fix $n\geq 1$. In the following, $\CN_n$ 
is the formal moduli space of $p$-divisible groups  of unitary type of signature
$(1,n-1)$,  that parametrizes tuples $(X, \iota, \lambda, \rho)$, where the quasi-isogeny $\rho$ is of
height zero, cf. \cite{KR}. Here is what we mean. 

Let $\breve F$ be the completion of the maximal unramified extension of $F$, with ring of integers $\CO_{\breve F}$ and residue field $\bar k$. We denote by ${\rm Nilp}={\rm Nilp}_{\CO_{\breve F}}$ the category of $\CO_{\breve F}$-schemes such that locally $\pi$ is a nilpotent element in the structure sheaf. We consider triples $(X, \iota, \lambda)$ where $X$ is a formal $\CO_F$-module of height $2n$, and $\iota: \CO_E\longrightarrow {\rm End}(X)$ is an action of $\CO_E$ on $X$ with Kottwitz condition of signature $(1, n-1)$, and where $\lambda$ is a  principal polarization whose associated Rosati involution induces the automorphism $\sigma$ on $\CO_E$. There is a unique such triple $(\BX, \iota, \lambda)$ over $\bar k$ such that $\BX$ is supersingular, up to $\CO_E$-linear isogeny preserving the polarizations up to a scalar.  We also write $\BX_n$ when we want to stress the dependence on $n$. Then $\CN_n$ represents the functor which to $S\in {\rm Nilp}$ associates the set of isomorphism classes of quadruples $(X, \iota, \lambda, \rho)$, 
where  $(X, \iota, \lambda)$ is a triple as 
 above over $S$, and where $\rho$ is an  $\CO_E$-linear quasi-isogeny $\rho: X\times_S\bar S\to\BX_n\times_{{\rm Spec}\, \bar k} \bar S$ of height zero, which carries the polarization on $\BX_n$ into one which differs locally by an element in $\CO_F^\times$ from $\lambda\times_S\bar S$. Here 
$\bar S=S\times_{{\rm Spec}\, \CO_{\breve F}} {\rm Spec}\, \bar k$.  

The functor $\CN_1$ is representable by ${\rm Spf}\, \CO_{\breve F}$, with universal object $(Y, \iota_0, \lambda_0)$.  We denote by $(\overline{Y}, \bar \iota_0, \bar \lambda_0)$ the same formal $\CO_F$-module as $Y$, but where $\bar \iota_0$ is obtained from $\iota_0 $ by pre-composing with $\sigma$. We may (and will) assume that for the framing objects for $\CN_{n-1}$, resp. $\CN_n$  we have the relation 
$$
\BX_n=\BX_{n-1}\times (\overline{Y}\times_{{\rm Spec}\, \CO_{\breve F}} {\rm Spec}\, \bar k) .
$$

For fixed $n\geq 2$, we abbreviate $\CN_{n-1}$ into $\CM$ and $\CN_{n}$ into  $\CN$.  We define the embedding 
\begin{equation}\label{delta}
\delta: \CM\hookrightarrow\CN, 
\end{equation}
via
$$
\delta\big((X, \iota, \lambda, \rho)\big)= (X\times \overline{Y}, \iota\times \bar \iota_0, \lambda\times \bar\lambda_0, \rho\times {\rm id}).
$$

Let 
$$
G=G_n=\{g\in {\rm End}_{E}^0(\BX_n)\mid \quad gg^\dag=1\}.
$$
Here $\dag$ is the Rosati involution induced by $\lambda$.
Then $G$ acts on $\CN_n$, by changing $\rho$ into $g\circ \rho$.

\section{The Bruhat-Tits stratification}
We recall some basic structure of the reduced part of the formal schemes of the last section,
especially the Bruhat-Tits stratification, comp.\, \cite{KR, V, VW}. This applies both to $\CM=\CN_{n-1}$ and to $\CN=\CN_n$. Let us explain the case of $\CM$.

We identify $E$ with the invariants of $\sigma^2$ in $\breve{F}$. Let $C_{n-1}$ be the hermitian space of dimension $n-1$ with hermitian form isomorphic to ${\rm diag}(1, \ldots ,1, p)$ (this differs by the factor $p$ from the form in \cite{V}).  Recall the concept of a {\it vertex lattice} in $C_{n-1}$: this is a lattice $\Lambda$ with $\pi\Lambda\subset\Lambda^*\subset\Lambda$, cf.\ \cite{KR}.  Here, as elsewhere in the paper,  $\Lambda^*$ denotes the dual lattice,  consisting of elements in the ambient vector space which pair integrally with all elements of $\Lambda$. The {\it type} of $\Lambda$ is the dimension of the $k'$-vector space $\Lambda/\Lambda^*$. 

  We denote by $\tau$ the automorphism ${\rm id}\otimes\sigma^2$ of $C_{n-1}\otimes_E {\breve F}$. We extend the hermitian form on $C_{n-1}$ to a sesqui-linear form on 
  $C_{n-1}\otimes_E {\breve F}$ by
$$
(x\otimes c, y\otimes c')=c\sigma(c')\cdot (x, y),\ x, y\in C_{n-1};\  c, c'\in {\breve F} . 
$$
The set $\CM(\BF)$ can be identified with the set  of lattices $A\subset
C_{n-1}\otimes_{E} {\breve F}$ such that
$$ A^\ast\subset^1 A\subset \pi^{-1}A^\ast ,
$$
where the notation ``$\subset^1$'' means that the quotient $A/A^\ast $ is a $\bar k$-vector space of dimension $1$. 

Recall \cite{KR} that to a lattice $A\in \CM(\BF)$, there is associated a vertex lattice $\Lambda=\Lambda (A)$ in $C_{n-1}$ via the following rule:
\begin{equation}
\Lambda(A)\otimes_{\CO_E}\CO_{\breve F}=\sum_{0}^{d}\tau^i A \text { is } \tau \text{-stable, for some } d.
\end{equation} 
 Then $\Lambda(A)\otimes_{\CO_E}\CO_{\breve F}$ is the smallest $\tau$-invariant lattice containing $A$. Dually, $\Lambda(A)^*\otimes_{\CO_E}\CO_{\breve F}$ is the largest $\tau$-invariant lattice contained in $A^*$. The
{\it type} of $A$ is the integer $t= t(A)=2d+1$, where $d$ is minimal. Equivalently, it is the type $t(\Lambda)$ of $\Lambda=\Lambda(A)$.

 For a given vertex lattice $\Lambda$, the lattices $A$ with $\Lambda=\Lambda(A)$ form the open Bruhat-Tits stratum $\CV_\Lambda(\BF)^o$ associated to $\Lambda$. The closed Bruhat-Tits stratum associated to $\Lambda$  is given by 
$$
\CV_\Lambda(\BF)=\{ A\in\mathcal M(\BF)\mid A\subset\Lambda\otimes_{\CO_E}\CO_{\breve F}\} .
$$
It turns out that these strata are in fact the $\BF$-points of algebraic subvarieties of $\CM_{\rm red}$, cf.\ \cite{VW}. More precisely, 
for any vertex $\Lambda$,  $\CV_\Lambda(\BF) $ is the set of $\BF$-points of a closed  
irreducible subvariety $\CV_\Lambda$ of $\CM_{\rm red}$ which is smooth of dimension
$\frac{1}{2}(t(\Lambda)-1)$,  the inclusions 
$ \CV_{\Lambda'}(\BF)\subset \CV_\Lambda(\BF)$ for $\Lambda'\subset \Lambda$
are induced by closed embeddings of algebraic varieties over $\BF$, and the open stratum $\CV_\Lambda(\BF)^o$ is the set of $\BF$-points of the open subvariety of $\CV_\Lambda$ obtained by removing all $\CV_{\Lambda'}$ for $\Lambda'\subsetneq \Lambda$.

Let $C_{n}=C_{n-1}\oplus E u$ with $(u,u)=1$. We again extend the pairing to a sesqui-linear pairing on  $C_n\otimes_E {\breve F}$. Then the preceding explanations apply to $\CN$ instead of $\CM$, and in particular 
$$ \CN(\BF)=\{B \mid B \text{ a lattice in }  C_n\otimes_E{\breve F} \text{ with } B^\ast\subset^1 B\subset \pi^{-1}B^\ast\} ,
$$
and again there is a Bruhat-Tits stratification, this time parametrized by  vertex lattices in $C_n$. The relation between $\CM$ and $\CN$ is given by the following lemma. 
\begin{lemma}\label{deltaonpoints}
The injection $\delta: \mathcal M(\BF)\to \mathcal N(\BF)$ induced on $\BF$-points by \eqref{delta} is given by $A\mapsto B=A\oplus \CO_{\breve F}u$. Furthermore, 
$$\Lambda(B)=\Lambda(A)\oplus \mathcal O_E u .$$
In particular, the types of $B$ and $A$ are the same. 
\end{lemma}
\begin{proof}
The first assertions follow easily from the identification of $\CN(\bar k)$ in terms of lattices. The last assertion is  obvious since $u$ is a unimodular vector. 
\end{proof}
\begin{corollary}
The morphism $\delta$ is compatible with the Bruhat-Tits stratifications of $\CM_{\rm red}$, resp.\ $\CN_{\rm red}$, in the sense that the stratum 
of $\CM_{\rm red}$ corresponding to the vertex lattice $\Lambda$ in $C_{n-1}$ is mapped to the stratum of $\CN_{\rm red}$ corresponding to the vertex lattice $\Lambda\oplus \mathcal O_E u$ in $C_n$. \qed
\end{corollary}

\section{An intersection problem}

 The morphism  $\delta$ induces a closed embedding of formal schemes, 
\begin{equation}
\Delta: \CM\to \CM\times_{{\rm Spf}\, \CO_{\breve F}}\CN,
\end{equation}
 with components ${\rm id}_\CM$ and $\delta$. Let $g\in G$. Then $g$ induces an automorphism $g: \CN\to\CN$. We denote by $\CN^g$ the fixed point locus,  defined to be  the intersection in $\CN\times_{{\rm Spf} \CO_{\breve F}}\CN$,
$$
\CN^g=\Delta_\CN\cap\Gamma_g . 
$$
Here $\Delta_\CN\subset \CN\times_{{\rm Spf} \CO_{\breve F}}\CN$ is the diagonal of $\CN$, and $\Gamma_g$ is the graph of $g$.

\begin{definition}
An element $g\in G$ is called regular semi-simple, if the matrix in ${ M} _n(E)$
$$
(g^iu, g^ku), i=0,\ldots, n-1; k=0,\ldots n-1\ ,
$$
is non-singular. Equivalently, the vectors $g^iu, i=0,\ldots,n-1, $ form a basis of $C_n$. 
\end{definition}

Here we have identified the group $G=G_n$ with the unitary group of $C_n$ as explained in \S 2.2 of \cite{Z}. 

This definition coincides with the definition of regular semi-simplicity in the introduction. Indeed,  we may identify $C_n$ with the hermitian space $(E^n, J)$ for $J=J_1\oplus 1$ such that $u$ is mapped to $v=(0,0,\ldots,0,1)$. If the vectors $g^iv, i=0,\ldots,n-1, $ form a basis of $C_n$, then so do the vectors $^tg^iv=J^{-1}\bar g^i Ju, i=0,\ldots,n-1$ since $Ju=u$. Hence $^tvg^{i}, i=0,\ldots,n-1,$ also form a basis of $E^n$.

\begin{proposition}\label{scheme}
 (i) There is an equality of formal schemes over ${\rm Spf}\, \CO_{\breve F}$
$$
\delta(\CM)\cap\CN^g=\Delta(\CM)\cap ({\rm id}_\CM\times g)\Delta(\CM) . 
$$
\smallskip
{\rm Now let $F=\BQ_p$.}

\smallskip

\noindent (ii)  If $g$ is regular semi-simple, then this formal scheme is a  scheme (i.e., 
any  ideal of definition is nilpotent) with underlying reduced subscheme proper over ${\rm Spec}\, \BF$. 

\smallskip 

\noindent (iii) If $g$ is regular semi-simple and $\big(\delta(\CM)\cap\CN^g\big)(\BF)$ is finite,  then 
$$
\CO_{\Delta(\CM)\cap ({\rm id}_\CM\times g)\Delta(\CM)}=\CO_{\Delta(\CM)}\otimes^\BL \CO_{({\rm id}_\CM\times g)\Delta(\CM)} ,
$$ 
i.e., the sheaf on the LHS represents the object on the RHS in the derived category.
\end{proposition} 
\begin{proof}
The first assertion follows by checking the equality on $S$-valued points, for $S\in {\rm Nilp}$, where it is a tautology. For the second assertion, we refer to \cite{Z}, Lemma 2.8.

For the third assertion, first note that if the intersection has a  finite number of points, it is an artinian scheme. Now both $\Delta(\CM)$ and $\CM\times_{{\rm Spf}\, \CO_{\breve F}}\CN$ are regular formal schemes of dimension $n-1$, resp. $2(n-1)$, and  therefore locally $\Delta(\CM)$ is the intersection of $n-1$ regular divisors in $\CM\times_{{\rm Spf}\, \CO_{\breve F}}\CN$. The same applies to $({\rm id}_\CM\times g)\Delta(\CM)$. Hence, if  $\Delta(\CM)\cap ({\rm id}_\CM\times g)\Delta(\CM)$ is discrete, the intersection of the $2(n-1)$ regular divisors is proper. Hence there are no higher Tor-terms, and the assertion follows, comp. \cite{T}. 
\end{proof}
\begin{remark}
The hypothesis $F=\BQ_p$ for (ii) and (iii) is made because the proof of \cite{Z}, Lemma 2.8 makes use of global methods. In fact, the proof uses a globalization of the special divisors $\CZ(x)$ of \cite{KR}. The assertions should be true for arbitrary $F$.

 It follows from (ii) that the Euler-Poincar\'e characteristic of $\CO_{\Delta(\CM)}\otimes^\BL \CO_{({\rm id}_\CM\times g)\Delta(\CM)}$ is finite. The {\it arithmetic intersection number} is defined to be 
\begin{equation}\label{defofarint}
\big\langle\Delta(\CM), ({\rm id}_\CM\times g)\Delta(\CM)\big\rangle=\chi( \CO_{\Delta(\CM)}\otimes^\BL \CO_{({\rm id}_\CM\times g)\Delta(\CM)}\big) {\rm log}\, q .
\end{equation}
\end{remark}
\begin{remark} In case the intersection $\delta(\CM)\cap\CN^g$ is discrete, it follows that locally at a point of intersection the fixed point locus $\CN^g$  is purely  formally one-dimensional: indeed, in this case the formal scheme $\CN^g$ intersects properly the formal divisor $\CM$ of $\CN$. 

\end{remark}
\begin{remark}
In case the intersection $\delta(\CM)\cap\CN^g$ is discrete, its underlying set  is stratified by the Bruhat-Tits stratification of $\CM$.  We define in this case for a vertex lattice $\Lambda$ in $C_{n-1}$
\begin{equation}\label{defmult}
{\rm mult}(\Lambda)=\sum_{x\in\CV(\Lambda)^o(\BF)} \ell\big({\CO_{\Delta(\CM)\cap ({\rm id}_\CM\times g)\Delta(\CM), x}}\big) . 
\end{equation}
The total arithmetic intersection number \eqref{defofarint} is in this case given by a finite sum 
$$
\big\langle\Delta(\CM), ({\rm id}_\CM\times g)\Delta(\CM)\big\rangle={\rm log}\, q \sum_{\Lambda} {\rm mult}(\Lambda) . 
$$
\end{remark}
Our next task will be to analyze which vertex lattices $\Lambda$ contribute effectively to this sum, and to understand the set of points in $\CV(\Lambda)^o\cap\CN^g$.

\section{Description of $(\delta(\CM)\cap \CN^g)^{}(\BF)$}

Let $g\in G$. It is clear that 
\begin{equation}
\CN^g(\BF)=\{B\in \CN(\BF)\mid  g(B)\subset B\}=\{B\in \CN(\BF)\mid  g(B)=B\}.
\end{equation}

\begin{lemma}
If $B\in \CV_\Lambda(\BF)^o$ is stable under $g$, then $g(\Lambda)=\Lambda$. In particular, if $\CN^g(\BF)\neq\emptyset$,  then 
the characteristic polynomial ${\rm char}_g(T)$ of $g$ has integral coefficients, i.e., ${\rm char}_g(T)\in \CO_E[T]$. 
\end{lemma}

\begin{proof}
Obvious, since $g$ and $\tau$ commute. 
\end{proof}

To a regular semi-simple element $g\in G$ we associate the  $\CO_E$-lattice $L_g$ in $C_n$ generated by $g^iu$ (an $\CO_E$-lattice, since $u, 
g u,\ldots, g^{n-1}u$ form a basis of $C_n$).

\begin{lemma}\label{squeeze}
Let $g\in G$  be regular semisimple.  Then for
any $B=A\oplus \CO_{\breve F}u$ stable under $g$, we  have
$$ L_g \subset
\Lambda(B)^\ast\subset \Lambda(B)\subset L_g^\ast.
$$
In particular, the invariants $(g^iu,g^{k}u)$ take  values in $\mathcal O_E$. 

Conversely, if $B\in \CN(\BF)$ contains $L_{g}$, then $B$ is of the form $B=A\oplus \CO_{\breve F}u$, for a unique $A\in \CM(\BF)$. 
\end{lemma}

\begin{proof}
Recall that  $\Lambda(B)^{\ast}$ is the largest
$\tau$-invariant lattice contained in $B^\ast$. Since $B$ is of the form $B=A\oplus \CO_{\breve F}u$, it follows that $u\in B^*$. Since $gB=B$, we also have $gB^*=B^*$. Hence $g^iu\in B^*$ for all $i\geq 0$. Hence $L_g\subset B^*$, and therefore $L_g\subset \Lambda(B)^*$ by the maximality of $\Lambda(B)^*$.
The other inclusion  is obtained by taking duals.

For the converse, note that the inclusion $L_g\subset B^*$ implies that $u\in B^*$. Since $u$ is a unimodular vector, setting $A=u^\perp$, we obtain $B=A\oplus \CO_{\breve F}u$. 
\end{proof}

By definition, $L=L_g$ is a $g$-cyclic lattice, i.e., there exists $v\in L$ such that $L$ is generated over $\CO_E$ by $\{ g^iv\mid 0\leq i\leq n-1\}$. 

\begin{lemma}\label{cycl} Let $g\in G$ and 
let $L$ be a $g$-cyclic lattice with $gL=L$. Then $L^*$ is also a $g$-cyclic lattice with $gL^*=L^*$. 
\end{lemma}
\begin{proof} Let $L$ be generated by $v_i=g^iv$, where $i=0,\ldots,n-1$. Then the $v_i$ form a basis of $C_n$. Let $v'_i$ be the dual basis, i.e.,
$$
(v_i, v'_j)=\delta_{ij} .
$$
Then $L^*$ is the $\CO_E$-span of $\{ v'_i\mid i=0,\ldots,n-1\}$. Let $v'=v'_0$. We claim that $v'$ is a cyclic generator of $L^*$, i.e., that the elements $w_i=g^iv'$ for $i=0,\ldots, n-1$ generate $L^*$ as an $\CO_E$-module.  It is clear that $w_i\in L^*$ for all $i$. 

{\bf Claim}: {\it $v'_j-w_j$ is a $\CO_E$-linear combination of $w_0,\ldots,w_{j-1}$, or equivalently, $v'_j-w_j$ is a $\CO_E$-linear combination of $v'_0,\ldots,v'_{j-1}$ (i.e., the matrix representing the base change from $v'_i$ to $w_i$ is a unipotent upper triangular matrix with integral entries). }

It is clear that this claim implies the lemma. Now for any vector $w$, we have $w=\sum\nolimits_i(w, v_i)v'_i$. Hence the claim is equivalent to 
\begin{equation}\label{inn}
(v'_j-w_j, v_i)=\begin{cases} \in \CO_E&\text 
{if $i<j$,}\\
0&\text{if $i\geq j$.}
\end{cases}
 \end{equation}
 Now for $i\geq j$, we have
 $$
 (w_j, v_i)=(g^jv', g^iv)=(v', g^{i-j}v)=\delta_{i j}=(v'_j, v_i) .
 $$
 This proves the second clause in (\ref{inn}). The first clause is trivial since $v'_j-w_j\in L^*$.  
 \end{proof}
\begin{definition} Let $L$ be a lattice in $C_n$ with $L\subset L^*$. Then set 
 $$
 {\rm Vert}(L)=\{\Lambda\mid \Lambda \text{ vertex lattice with } L\subset \Lambda^*\subset\Lambda\subset L^* \} , 
$$
cf.\ \cite{KR}. If $g\in G$ with $gL=L$, then $g$ acts on ${\rm Vert}(L)$, and we set 
$$
{\rm Vert}^g(L)=\{\Lambda\in{\rm Vert}(L)\mid g\Lambda=\Lambda \} .
$$
\end{definition} 
Note that by Lemma \ref{squeeze}, the assumption on $L$ is satisfied for $L=L_g$, if $\delta(\CM)\cap \CN^g\neq\emptyset$. We may summarize Lemma \ref{squeeze} as follows.  
\begin{corollary}\label{stratCZ}
(i) If $g$ is regular semi-simple, there is an equality of sets $$(\delta(\CM)\cap \CN^g)(\BF)=\{B\in \CN(\BF)\mid {L_g}\otimes_{\CO_E}\CO_{\breve F}\subset B^*\subset B\subset {L_g^*}\otimes_{\CO_E}\CO_{\breve F}, gB=B\} .$$

\smallskip

\noindent (ii) There is an equality of $\bar k$-varieties $$\big(\delta(\CM)\cap \CN^g\big)_{\rm red}= \bigcup_{\Lambda\in {\rm Vert}^g(L_g)}\big(\CV(\Lambda)^o\big)^{g_\Lambda} .$$
\end{corollary}
Here (ii) makes use of Proposition \ref{scheme}, and the algebraicity of the Bruhat-Tits stratification. The action of $g$ on $\CV(\Lambda)^o$ is induced by the automorphism   $g_\Lambda$ on $\Lambda/\Lambda^*$ induced by $g$. 

\section{Fixed point set in a stratum}\label{fpinstratum}

We next analyze the fixed point variety of $g_\Lambda$ on $\CV(\Lambda)^o$. Since $\CV(\Lambda)^o$ is a Deligne-Lusztig variety \cite{VW}, this can be considered as a general question on Deligne-Lusztig varieties (called ``DL-varieties'' below for brevity). Accordingly, we use notation that is standard in this context, e.g., \cite{DL}. 
\begin{proposition}\label{fixpointstr}
 Let $\Lambda$ be a vertex lattice in $C_n$ with 
$g\Lambda=\Lambda$, and denote by $\bar g=g_\Lambda\in {\rm U}(V)(\BF_p)$ the automorphism of the hermitian space $V=V_\Lambda=\Lambda/\Lambda^*$ over $k'$ induced by $g_\Lambda$. 

\noindent (i) If $\big(\CV(\Lambda)^o\big)^{\bar g}$ is non-empty, then $g_\Lambda$ is semi-simple and contained in a Coxeter type maximal torus.  

\noindent (ii) If $\big(\CV(\Lambda)^o\big)^{\bar g}$ is a non-empty finite set, then $g_\Lambda$ is a regular elliptic element contained in a Coxeter type maximal torus. Furthermore, in this case the cardinality of $\big(\CV(\Lambda)^o\big)^{\bar g}$ is given by the type of $\Lambda$. 
\end{proposition}
\begin{proof} This follows from the following lemmas on DL-varieties.
\end{proof}

We first  recall that an element $w$ in the Weyl group $W$ is called {\it elliptic}, if the following equivalent properties are satisfied:
\begin{enumerate}
\item [(i)] The torus $T_w$ of type $w$ is elliptic, i.e., $T_w/Z$ is anisotropic (i.e., $X^*(T_w/Z)^F=(0)$). 
\item[(ii)] $T_w$ is not contained in a proper $F$-stable parabolic subgroup. 
\item[(iii)] $1$ is not an eigenvalue of $w\cdot F_*$. 
\item[(iv)] The $F$-conjugacy class of $w$ contains no element in a proper $F$-stable parabolic subgroup of $W$. 
\end{enumerate}
Here $F_*$ denotes the action of Frobenius on $X_*(T),$ where $T$ is a maximal torus contained in a Borel subgroup (if $G$ is split, then $F_*$ is trivial;  for a unitary group in $n$ variables with standard basis for the hermitian space, $F_*$ acts through  the longest element in $S_n$). Note that any Coxeter element in the sense of Lusztig \cite{Lu1} is elliptic. 

(The equivalence of (i) and (iii) follows from \cite{C}, Proposition 3.2.2. The equivalence of (iii) and (iv) follows from \cite{H}, Lemma 7.2. The equivalence of (i) and (ii) is easy\footnote{ We thank X. He for pointing out these references.}.)

The DL-varieties appearing in \cite{VW} are  associated to  unitary groups in an odd number of variables and  standard Coxeter elements. More precisely, let $V$ be a hermitian vector space over $\BF_{q^2}$ of dimension $n=2d+1$. We choose the basis $e_1,\ldots, e_n$ in such a way that under the hermitian pairing $e_i$ pairs trivially with $e_j$, unless $i+j=n+1$, and we identify $W$ with the symmetric group $S_n$. Then the DL-variety of interest is associated to the cyclic permutation $w=(d+1, d+2,\ldots, n)$. The DL-varieties associated to different Coxeter elements all differ at most by a power of Frobenius \cite{Lu1}, Prop.~1.10; in particular, they are all universally homeomorphic. The DL-variety $X_w$ associated to the Coxeter element $w=(1, 2,\ldots,d+1)$ is the variety  of complete flags $\CF_\bullet$ such that
\begin{equation*}
\CF_{n-i}^\perp\subset \CF_{i+1}, \, \CF_{n-i}^\perp\neq \CF_i,\, (1\leq i\leq d);\,\, \CF_i=\CF_{n-i}^\perp,\,  (d+1\leq i\leq n-1) .
\end{equation*}
Let $\tau=\sigma^2$. Then $X_w$ can   also be identified with the variety  of complete selfdual flags $\CF_\bullet$ of $V$ such that 
\begin{equation*}
\CF_i+\tau(\CF_i)=\CF_{i+1} , \, i=1,\ldots,d .
\end{equation*}
 In other words, $X_w$ is the variety of complete isotropic flags $\CF_1\subset\CF_2\subset\ldots\subset\CF_d$ of $V$ such that
\begin{equation*}
\CF_1\neq \tau(\CF_1)\subset \CF_2, \CF_2\neq \tau(\CF_2)\subset \CF_3,\ldots, \CF_{d-1}\neq \tau(\CF_{d-1})\subset \CF_d, \CF_d\neq \tau (\CF_d) .   
\end{equation*} Hence we can identify $X_w$ with the set of $\ell\in\BP(V)$ such that 
\begin{equation*}
\big(\ell, \ell\big)=\big(\ell, \tau(\ell)\big)=\ldots =\big(\ell, \tau^{d-1}(\ell)\big)=0;\, \big(\ell, \tau^{d}(\ell)\big)\neq 0 .
\end{equation*}
This DL-variety is defined over $\BF_{q^2}$.

\begin{lemma}{\rm(Lusztig \cite[5.9]{Lu2})}\label{semisimple}
Let $X_w$ be a DL-variety, where $w$ is elliptic and of minimal length in its $F$-conjugacy class. Let $s\in G(\BF_q)$. If the fixed point set $X_w^s$ is non-empty, then $s$ is semi-simple.\qed
\end{lemma}

\begin{proof} (Lusztig) In the case of a unitary group in an odd number of variables and   the standard Coxeter element, this can be easily seen as follows. In this case, as explained above, $X_w$ can be viewed as a subset of projective space, by associating to a complete flag its one-dimensional component $\ell\subset V\otimes_{\BF_{q^2}}\BF$. Now assume that $\ell$ is fixed under $s$. Then so are $\tau(\ell), \tau^2(\ell),\ldots$. But if $\ell\in X_w$, then $\ell, \tau(\ell), \ldots, \tau^{n-1}(\ell)$ form a basis of $V\otimes_{\BF_{q^2}}\BF$, cf. \cite{Lu3}, Prop.~26, (i). Hence $s$ is a diagonal element wrt this basis. 
\end{proof}

\begin{lemma}
Let $X_w$ be a DL-variety, and let $s\in G(\BF_q)$ be a semi-simple element. Then the fixed point set $X_w^s$ is non-empty
 if and only if $s$ is conjugate under $G(\BF_q)$ to an element in $T(\BF_q)$ for a maximal torus $T$ of type $w$. 
\end{lemma}
\begin{proof}
This follows immediately from the formula (4.7.1) for $X_w^s$ in \cite{DL}, Prop. 4.7. 
\end{proof}

\begin{remark}
We know from \cite{DL} that the fixed point set is a finite disjoint sum of DL-varieties, for various groups and various Weyl group elements. 
Let us spell out which DL-varieties occur in the case of interest to us, namely the unitary group of odd size $n=2d+1$, and when $w=(1, 2,\ldots,d+1)$ is the Coxeter element as above. Now in this case the maximal torus $T$ of type $w$ is given by 
$$
T(\BF_q)={\rm Ker \, \big(Nm}_{\BF_{q^{2n}}/\BF_{q^n}}: {\BF^\times_{q^{2n}}}\to{\BF^\times _{q^n}}\big).  
$$
We may identify the hermitian space $V$ with $\BF_{q^{2n}}$, equipped with the hermitian form 
$(x, y)\mapsto {\rm Tr}_{\BF_{q^{2n}}/\BF_{q^2}}(\sigma^n(x) y)$. Now $s\in T(\BF_q)$. Hence $s$ generates a subfield $\BF_q(s)$ of 
$\BF_{q^{2n}}$. Let $\BF_q(s)=\BF_{q^h}$. Then $h\vert 2n$. If $h$ is odd, then the norm equation for $s$ gives $s^2=1$, hence $h=1$ and $s=\pm 1$, and $s$ acts trivially on $X_w$. If $h=2k$ is even, then $k\vert n$. In this case, we may identify the hermitian space $V$ with $\BF_{q^{2n}}$, equipped with the hermitian form 
$(x, y)\mapsto {\rm Tr}_{\BF_{q^{2k}}/\BF_{q^2}}\big({\rm Tr}_{\BF_{q^{2n}}/\BF_{q^{2k}}}(\sigma^n(x) y)\big) .$
Then the centralizer $Z^0(s)$ can be identified with 
$$
Z^0(s)={\rm Res}_{\BF_{q^k}/\BF_q}(U_h) ,
$$
where $U_h$ is the unitary group for the hermitian form ${\rm Tr}_{\BF_{q^{2n}}/\BF_{q^{2k}}}(\sigma^n(x) y)$ on $\BF_{q^{2n}}$, and the maximal torus can be identified with the restriction of scalars of the maximal torus ${\rm Ker\, Nm}_{\BF_{q^{2n}}/\BF_{q^n}}$ of $U_h$. In this case the corresponding DL-variety is simply the DL-variety of dimension $\frac{1}{2}(\frac{n}{k}-1)$ associated to the Coxeter torus in a unitary group of odd size $\frac{n}{k}$ over $\BF_{q^k}$.  And the fixed point set $X_w^s$ is a disjoint sum of isomorphic copies of this DL-variety.

\end{remark}

\begin{lemma}
Let $w$ be elliptic, and $s\in G(\BF_q)$ semi-simple. If $X_w^s$ has only finitely many elements, then $s$ is regular, and conversely. 
\end{lemma}
\begin{proof}
We use the formula (4.7.2) for $X_w^s$ in \cite{DL}, Prop. 4.7., which presents $X_w^s$ as a disjoint union
of varieties which are DL-varieties for $Z^0(s)$, of the form $X_{T'\subset B'}$. However, $T'$ is of the same type as $T$, hence is elliptic. On the other hand, if the fixed point set is finite, then ${\rm dim}\,  X_{T'\subset B'}=0$. This implies that $T=Z^0(s)$, which is precisely the claim.  The converse is obvious, because a regular element has only finitely many fixed points in the flag variety. 
\end{proof}
\begin{lemma}\label{cardcentr}
Let $s\in G(\BF_q)$  be regular and contained in a maximal torus $T$ of type $w$. Then the number of fixed points of $s$ in $X_w$ is equal to the cardinality of the $F$-centralizer of $w$ in $W$.
\end{lemma}
\begin{proof}
We use the formula (4.7.1) in \cite{DL}, Prop. 4.7. It shows that the cardinality of $X_w^s$ is equal to the cardinality of $N(\BF_q)/T(\BF_q)$, where $N$ denotes the normalizer of $T$. However $N(\BF_q)/T(\BF_q)$ can be identified with the fixed points under the action of Frobenius on $N/T$. After identifying $N/T$ with $W$, this action is via $x\mapsto wF(x)w^{-1}$. Hence the fixed points are identified with the $F$-centralizer of $w$.
\end{proof}
\begin{lemma}\label{cardfix}
Let $G$ be the unitary group in an odd number $n$ of variables. Then 
\smallskip

\noindent (i) The $F$-centralizer of a Coxeter element $w$ has $n$ elements.

\smallskip

\noindent (ii) Let $s\in T(\BF_q)$ be a regular element in a Coxeter torus. Then all points in $X_w^s$ are conjugate under 
${\rm Gal}(\BF_{q^{2n}}/\BF_{q^2})$, and in fact, this Galois group acts simply transitively on the fixed points. 
\end{lemma}
\begin{proof}
By Lemma \ref{cardcentr},  the second assertion implies the first one, since ${\rm Gal}(\BF_{q^{2n}}/\BF_{q^2})$ has $n$ elements. Now we may identify $T$ with ${\rm Ker}({\rm Nm}_{\BF_{q^{2n}}/\BF_{q^n}}) $, and the hermitian space $V$ with $\BF_{q^{2n}}$, cf. above. Then the set of fixed points of a regular  $s\in T(\BF_q)$ 
in $\BP(V)$ is just the set of eigenlines of $\BF_{q^{2n}}^\times$ in $V\otimes_{\BF_{q^2}}\BF$. These all lie in $X_w$, and this implies the assertion. 
\end{proof}
At this point all statements of Proposition \ref{fixpointstr} are proved. We also note the following consequence. 
\begin{corollary}
Let $\Lambda\in{\rm Vert}^g(L)$ such that $\big(\CV(\Lambda)^o\big)^{g_\Lambda}$ is finite. Then there is no $\Lambda'\in {\rm Vert}^g(L)$, with  
$\Lambda'$  strictly contained in $\Lambda$. 
\end{corollary}
\begin{proof}
Indeed, $\Lambda'$ would correspond to a proper parabolic in ${\rm U}(V_{\Lambda})(\BF_p)$; but $ g_\Lambda$ is not contained in a proper parabolic by Proposition \ref{fixpointstr}, and hence cannot fix $\Lambda'^*/\Lambda^*$. 
\end{proof}

\section{Statement of the AFL}


Let $C'_n$ be a hermitian space of dimension $n$ with discriminant of \emph{even} valuation, and equipped with a vector $u$ of norm one. We fix a self-dual lattice $L_0$ in $C'_n$ such that $u\in L_0$. We denote by $K$ the stabilizer of $L_0$ in $U(C_n')(F)$. We define, for $g\in U(C_n')(F)$ regular semi-simple,
\begin{equation}\label{orbintu}
O(g,1_K)=\int_{U(u^\perp)(F)}1_K(h^{-1} g h)dh ,
\end{equation}
where the Haar measure is normalized by ${\rm vol}\big(K\cap {U(u^\perp)(F)}\big)=1$. Here  $u^\perp$ denotes  the orthogonal complement of $u$ in $C_n'$.

We now denote by $C$ either $C_n$ or $C_n'$.
For $g\in U(C)(F)$ regular semi-simple, we denote by $L=L_g$ the lattice in $C$ generated by $g^i u,i=0,1,\ldots,n-1$. We define an involution  $\tau$ on $C=L_g\otimes_{\CO_E} E$ (depending on $g$) by requiring that $(a \cdot g^iu)^\tau= \bar a\cdot g^{-i}u$ for $a\in E$ and $i=0,1,\ldots,n-1$.

\begin{lemma}\label{lem Cn'}
Let $g\in U(C) (F)$ be regular semisimple. Then 
$$
O(g,1_K)=\sum_{\{\Lambda\mid L\subset \Lambda\subset L^*, g\Lambda=\Lambda, \Lambda^*=\Lambda\}}1.
$$
\end{lemma}

\begin{proof}

The orbital  integral \eqref{orbintu} counts the number of self-dual lattices $\lambda$ in $u^\perp$  such that $\lambda\oplus \CO_E u$ is fixed by $g$.
To show the equality, it suffices to show that any  lattice $\Lambda$ occuring on the RHS  splits as a direct sum $\lambda\oplus \CO_E u$ for a self-dual lattice $\lambda$. But since $(u,u)=1$ it follows that  $\Lambda=(\Lambda\cap u^\perp)\oplus\CO_E u$, where $\lambda=\Lambda\cap u^\perp$ is self-dual.

\end{proof}

Now let $S_n$ be the variety over $F$ whose $F$-points are
$$
S_n(F)=\{s\in GL_n(E)\mid s\bar s=1.\}
$$
In fact, $S_n$ is defined over $\CO_F$. For $\gamma\in S_n(F)$, recall that its invariants are the characteristic polynomial ${\rm char}_\gamma(T)\in E[T]$ and the $n-1$ elements $v \gamma^i\, ^t\! v,i=1,2,\ldots,n-1$ of $E$, for $v$ the row vector $(0,\ldots,0,1)$. 

For $\gamma\in S_n(F)$ regular semi-simple and $s\in\BC$, we consider
$$
O(\gamma, 1_{S_n(\CO_F)},s)=\int_{GL_{n-1}(F)}1_{S_n(\CO_F)}(h^{-1}\gamma h)\eta({\rm det}\,h)|{\rm det}\, h|^sdh ,
$$
where the Haar measure on $GL_{n-1}(F)$ is normalized by ${\rm vol}(GL_{n-1}(\CO_F))=1$. This is a polynomial in $\BZ[q^s, q^{-s}]$, comp. Lemma \ref{lem Mi} below. 

We will simply denote the value at $s=0$ by $O(\gamma, 1_{S_n(\CO_F)})$;  it is given by\begin{equation}\label{orbintS}
O(\gamma, 1_{S_n(\CO_F)})=\int_{GL_{n-1}(F)}1_{S_n(\CO_F)}(h^{-1}\gamma h)\eta({\rm det}\, h)dh.
\end{equation}
We also introduce the first derivative at $s=0$:
\begin{equation}\label{orbintder}
O'(\gamma, 1_{S_n(\CO_F)})= \frac{d}{ds}O(\gamma, 1_{S_n(\CO_F)},s)_{\big| s=0} . 
\end{equation}

For a regular semisimple $\gamma\in S_n(F)$, we define $\ell(\gamma)=v({\rm det}(\gamma^i v))\in \BZ$, where $(\gamma^i v)$ is the matrix $(v,\gamma v,\ldots,\gamma^{n-1} v)$. And we define a sign 
\begin{equation*}
\omega(\gamma)=(-1)^{\ell(\gamma)} \in \{\pm 1\}.
\end{equation*}

Now let $g\in U(C)(F)$ match $\gamma$, i.e., have the same invariants as $\gamma$. Then, with $L=L_g$, 
we define the set of lattices in $C$
$$
M=\{\Lambda\mid L\subset \Lambda\subset L^*, g\Lambda=\Lambda,\Lambda^\tau=\Lambda\}
$$
and its subsets indexed by $i\in \BZ$,
$$
M_i=\{\Lambda\mid\Lambda\in M,\ell(\Lambda/L)=i\}.
$$
Here $\ell(\Lambda/L)$ is the length of the $\CO_E$-module $\Lambda/L$. 

\begin{lemma}
\label{lem Mi}
Let $GL_{n-1}(F)_i$ be the open subset of $GL_{n-1}(F)$ consisting of $h$ with $v({\rm det}\, h)=i$. Then 
$$\int_{GL_{n-1}(F)_i}1_{S_n(\CO_F)}(h^{-1}\gamma h)dh =| M_{i-\ell(\gamma)}|.$$
\end{lemma}

\begin{proof} Consider the row vector space $F^{n}$, with $F^{n-1}$ as a natural subspace (vectors with zero last entry). We also consider $v=(0,\ldots,0,1) $ as an vector in 
$F^{n}$. 
Consider the set  of lattices  
$$
 \mathfrak{m}:=\{\lambda\subset F^{n-1}\mid \gamma(\Lambda)=\Lambda, \text{ where } \Lambda=(\lambda\otimes \CO_E)\oplus\CO_E v\} 
$$
and the subsets 
$$
 \mathfrak{m}_i:=\{\lambda\in 
 \mathfrak{m}\mid \ell(\lambda/\lambda_0)=i\} ,\quad \lambda_0=\CO_F^{n-1}.
$$
Here the length is defined by $\ell(\lambda/\CO_F^{n-1}):=\ell(\lambda/A)-\ell(\CO_F^{n-1}/A)$ for any lattice $A\subset \lambda\cap \CO_F^{n-1}$. It is obvious that the LHS in Lemma \ref{lem Mi} is given by the cardinality $|\mathfrak{m}_i|$.

Denote by $\sigma$ the Galois conjugation on $E^n$. Define a hermitian form on $E^n$ by requiring that 
$$
(\gamma^i v,\gamma^j v):=v \gamma^{i-j}\,^t\!v.
$$
 Let $\CL=\CL_\gamma$ be the $\CO_E$-lattice in $E^n$ generated by $\gamma^i v, i=0,1,\ldots,n-1$.  Denote by $\CL^\ast$ the dual of $\CL$, i.e.,
 $$
\CL^\ast=\{ x\in E^n\mid (x, \CL)\subset \CO_E \} . 
$$
Now we introduce the set of lattices
 $$ \mathfrak{m}':=\{\Lambda\subset E^{n}\mid  \CL\subset \Lambda\subset \CL^\ast, \gamma \Lambda=\Lambda, \Lambda^\sigma=\Lambda\} ,$$
 and 
  $$
 \mathfrak{m}'_i:=\{\Lambda\in 
 \mathfrak{m}'\mid \ell(\Lambda/\CL)=i\} .
$$
We claim that the map $\lambda\mapsto \Lambda:=(\lambda\otimes \CO_E)\oplus\CO_Ev$ defines a bijection between $\mathfrak{m}$ and $ \mathfrak{m}'$. First of all, such $\Lambda$ do lie in $ \mathfrak{m}'$. Indeed, we  only  need to verify that  $\Lambda\subset \CL^\ast$ or, equivalently, $(\Lambda,\gamma^i v)\in \CO_E$ for all $i$. This  follows from $\gamma\Lambda=\Lambda$ and $(\Lambda, v)\in \CO_E$. Now we only need to show the surjectivity of the map.  Similarly to the unitary case, any $\Lambda\in  \mathfrak{m}'$ is a direct sum 
$ (\Lambda\cap E^{n-1})\oplus\CO_E v$. Obviously $ \Lambda\cap E^{n-1}$ is also invariant under the Galois conjugation  on $E^{n-1}$. So we may find a lattice $\lambda\subset F^{n-1}$ such that $ \lambda\otimes \CO_E=\Lambda\cap E^{n-1}$. This proves the surjectivity. 

We claim that the set $\mathfrak{m}_i$ is sent to $\mathfrak{m}'_{i-\ell(\gamma)}$. 
Clearly we have $\ell(\lambda/\lambda_0)=\ell(\Lambda/\CO_E^n)$ under this map.  Hence the claim follows, since the length of $\CL$ over the image of $\CO_E^{n}$ is obviously given by $\ell(\gamma)$.  

To finish the proof, we need to exhibit a bijection 
from $ \mathfrak{m}'$ to $M$ that sends $\mathfrak{m}'_i$ to $M_{i}$. Since $(g^iu,u)=v\gamma^i\, ^t\!v$ for all $i$, the map $\gamma^iv\mapsto g^i u$ defines an isometry between $\CL=\CL_\gamma$ and $L=L_g$. Moreover, the involution $\sigma$ on $E^n$ maps $\gamma^i v$ to $\bar \gamma^i v=\gamma^{-i}v$. Therefore $\sigma$ transfers to the involution $\tau$ on  $L\otimes E$. Clearly this map sends $\mathfrak{m}'_i$ to $M_{i},$ since it sends $\CL$ to $L$.

\end{proof}

\begin{corollary}\label{cor lattice counting}
Let $\gamma\in S_n(F)$ be regular semisimple and 
match  $g\in U(C)(F)$.
\begin{itemize}
\item[(1)] If $C=C_n'$, then
$$
O(\gamma,1_{S_n(\CO_F)})=\omega(\gamma)\sum_{\{\Lambda\mid L\subset \Lambda\subset L^*, g\Lambda=\Lambda,\Lambda^\tau=\Lambda\}} (-1)^{\ell(\Lambda/L)} .
$$

\smallskip

\noindent \item[(2)] If $C=C_n$, then  $O(\gamma, 1_{S_n(\CO_F)})=0$
and
$$
O'(\gamma, 1_{S_n(\CO_F)})=-\omega(\gamma){\rm log}\, q\sum_{\{\Lambda\mid L\subset \Lambda\subset L^*, g\Lambda=\Lambda,\Lambda^\tau=\Lambda\}} (-1)^{\ell(\Lambda/L)} \ell(\Lambda/L) . 
$$
\end{itemize}
\end{corollary}

\begin{proof} By Lemma \ref{lem Mi},
we have
$$
O(\gamma,1_{S_n(\CO_F)},s)=\sum_{i\in \BZ}(-1)^i  |M_{i-\ell(\gamma)}|q^{-is}.
$$
Or equivalently
$$
O(\gamma,1_{S_n(\CO_F)},s)=(-1)^{\ell(\gamma)}\sum_{i\in \BZ}(-1)^i  |M_{i}|q^{-(i+\ell(\gamma))s}.
$$
This shows that 
$$
O(\gamma,1_{S_n(\CO_F)})=(-1)^{\ell(\gamma)}\sum_{i\in \BZ}(-1)^i  |M_{i}|.
$$
In particular, if we set $C=C_n'$, the first identity is proved. 

Now let $C=C_n$. The map $\Lambda \mapsto \Lambda^\ast$ defines an involution on $M$ and sends $M_{i}$ to $M_{r-i}$ where $r$ is the length of $L^\ast/L$, which is odd. This shows that $O(\gamma, 1_{S_n(\CO_F)})=0$. We now take the first derivative 
\begin{align*}
O'(\gamma,1_{S_n(\CO_F)},0)&=-(-1)^{\ell(\gamma)}{\rm log}\, q \sum_{i\in \BZ}(-1)^i  |M_{i}|(i+\ell(\gamma))\\&=-\omega(\gamma) {\rm log}\, q \sum_{i\in \BZ}(-1)^i  i|M_{i}|.
\end{align*}
This completes the proof.
\end{proof}

Using Lemma \ref{lem Cn'} and Corollary \ref{cor lattice counting} above,  the  statement of the FL (cf.\ Introduction) is  the following identity for $g\in U(C_n')(F)$ regular semi-simple:
\begin{equation}
\sum_{\{\Lambda\mid L\subset \Lambda\subset L^*, g\Lambda=\Lambda,\Lambda^\tau=\Lambda\}} (-1)^{\ell(\Lambda/L)} = \sum_{\{\Lambda\mid L\subset \Lambda\subset L^*, g\Lambda=\Lambda, \Lambda^*=\Lambda\}}1. 
\end{equation}

Now, in the special case that the intersection of $\Delta(\CM)$ and $({\rm id}_\CM\times g)\Delta(\CM)$ is discrete, the statement of the AFL (cf.\ Introduction) is as follows.
\begin{conjecture}Let $g\in U(C_n)(F)$ be regular semi-simple.
Assume that $(\delta(\CM)\cap \CN^g)(\BF)$ is finite. Then 
$$
\sum_{\{\Lambda\mid L\subset \Lambda\subset L^*, g\Lambda=\Lambda, \pi\Lambda\subset\Lambda^*\subset\Lambda\}}{\rm mult}(\Lambda)=-\sum_{\{\Lambda\mid L\subset \Lambda\subset L^*, g\Lambda=\Lambda,\Lambda^\tau=\Lambda\}} (-1)^{\ell(\Lambda/L)} \ell(\Lambda/L) . 
$$
Here the number ${\rm mult}(\Lambda)$ is the intersection multiplicity of $\Delta(\CM)$ and $({\rm id}_\CM\times g)\Delta(\CM)$ {\it along the stratum $\CV(\Lambda)^o$}, cf. \eqref{defmult}.
\end{conjecture}

\section{The minuscule case}\label{theminusculecase}

In this section we assume that $V=L^\ast/L$ is killed by $\pi$. We thus consider it as a vector space over $k'=\BF_{q^2}$, the residue field of $\CO_E$.  Denote by $r$ its dimension. Then $r$ is an odd integer between $1$ and $n-1$. Then the hermitian form on $L$ naturally induces a non-degenerate hermitian form on $V$. (It is obtained as follows: For $\bar x, \bar y \in V$ with representatives $x,y\in L^*$, the value $(\bar x, \bar y)_V$ of the hermitian form on $V$ is the image modulo $(\pi)$ of $\pi \cdot (x,y)\in \CO_E$, where $(\ , \ )$ denotes the form on $L\otimes E$.) We denote the corresponding unitary group by $U(V)$ and consider it as an algebraic group defined over $\BF_q$.
As $g$ defines automorphisms of both $L$ and $L^\ast$, it induces an automorphism $\bar g\in U(V)$. Then via the map $\Lambda \mapsto \Lambda^*/L $ the set of vertices $\Lambda\in {\rm Vert}^g(L)$ is in natural bijection with the set of
$\bar g$-invariant $k'$-subspaces $W$ of $V$ such that $W$ is totally isotropic with respect to the hermitian form on $V$.  We thus define  ${\rm Vert}^{\bar g}(V)$ to be the set of such $W$.
And we write $\CV_W$ for the closed Bruhat-Tits stratum  $\CV_\Lambda$ that corresponds to $\Lambda\in {\rm Vert}^g(L)$ in the sense of  Vollaard's paper \cite{V}, cf.\ also \cite{VW}, and we call {\it type} of $W$ the type of $\Lambda$, i.e., the dimension of $W^\perp/W$. The open stratum $\CV_W^\circ$ can then be identified with the Deligne-Lusztig variety associated to a Coxeter torus of $U(W^\perp/W)$, cf. \cite{V}.

We will consider the characteristic polynomial $P_{\bar g|V}(T):={\rm det}(T-\bar g|V)\in k'[T]$ of degree $r$. Since, by Lemma \ref{cycl}, $L^\ast$ is  $g$-cyclic, $V$ is $\bar g$-cyclic. This is equivalent with the regularity of $\bar g$ as an endomorphism of $V$. In particular, its characteristic polynomial is equal to its minimal polynomial.  Let 
\begin{align}\label{prodirr}
P_{\bar g}=\prod_{i=1}^{\ell} P_i^{a_i} 
\end{align}
be the decomposition into irreducible monic polynomials.  

If $P(T)=T^d+b_1T^{d-1}+\ldots+b_d \in  k'[T],b_d\neq 0$, we set
$$
P^\ast(T)= \overline{b}_{d}^{-1}T^d \overline{P}(T^{-1}) ,
$$
where the bar denotes the Galois conjugate on $k'$.
Since $\bar g\in U(V)$,  we have $P_{\bar g}=P^\ast_{\bar g}=\prod_i P_i^{\ast a_i}$, and hence we have an involution $\tau$ of $\{1,2,\ldots,\ell\}$ such that $P_{i}^\ast=P_{\tau(i)}$ and $a_{\tau(i)}=a_i$, cf. \cite{AG}. Note that since $V$ has odd dimension, the degree of $P_{\bar g}$ is odd, and hence there exists at least one index $i$ with $\tau(i)=i$ and such that $a_i$ is odd.

\begin{prop}\label{cardform}
(i)  The  set $(\delta(\CM)\cap \CN^g)(\BF)$ is non-empty if and only if there exists a unique $i_0\in \{1,2,\ldots,\ell\}$ such that $\tau(i_0)=i_0$ and such that $a_{i_0}$ is odd.  Then the set  $(\delta(\CM)\cap \CN^g)(\BF)$  is finite.

\smallskip
 
 \noindent(ii) If $(\delta(\CM)\cap \CN^g)(\BF)$ is non-empty (hence finite),  these points lie on some strata  $\CV^\circ_W$, all  of the same type ${\rm deg}\, P_{i_0}$  for the unique $i_0$ in part $(i)$. And the cardinality of $(\delta(\CM)\cap \CN^g)(\BF)$ is given by
$$
\prod_{\{i,j\},j=\tau(i)\neq i} (1+a_i)\cdot {\rm deg}\,P_{i_0}.
$$
\end{prop}

\begin{proof}

If  $(\delta(\CM)\cap \CN^g)(\BF)$ is non-empty, then there exists  $W\in   {\rm Vert}^{\bar g}(V)$.
Then $W^\perp$ is $\bar g$-invariant and the hermitian form allows us to identify $W$ with the dual of $V/W^\perp$. This yields a filtration
$$
0\subset W\subset W^\perp \subset V
$$and a decomposition
\begin{align}\label{decomp}
P_{\bar g|V}=P_{\bar g|W}\cdot P_{\bar g|W^\perp/W}\cdot P_{\bar g|V/W^\perp},\quad
\end{align}
with the property
\begin{align}\label{propdecomp}
 P_{\bar g|W^\perp/W}=P_{\bar g|W^\perp/W}^\ast,\,\, P_{\bar g|W}=P^\ast_{\bar g|V/W^\perp}.
\end{align}
Note that the fixed point set  $\CV_W^{\circ,\bar g}$ is non-empty if and only if $\bar g|(W^\perp/W)$ lies in a Coxeter torus, i.e., (because we are dealing here with a unitary group in an {\it odd} number of variables), if and only if $\bar g|(W^\perp/W)$ generates inside ${\rm End}(W^\perp/W)$ a subfield of $\BF_{q^{2r}}$.  Since $\bar g$ is a regular endomorphism, so is the induced endomorphism on $W^\perp/W$.  Hence
 the fixed point set  $\CV_W^{\circ,\bar g}$ is non-empty if and only if  $P_{\bar g|W^\perp/W}$ is an irreducible polynomial. 
This irreducible polynomial has to be of the form $P_{i_0}$ with $\tau(i_0)=i_0$.  Moreover if $P_i|P_{\bar g|W}$, then $P_i^\ast|P_{\bar g|V/W^\perp}$. This shows that $a_{i_0}$ is odd and that for every $j\neq i_0$, either $\tau(j)\neq j$ or $\tau(j)=j$ and $a_j$ is even. This shows the ``only if" part of $(i)$. 
Moreover, the type of $W$, i.e., the dimension of $W^\perp/W$,  is equal to the degree of $P_{i_0}$. 

We now assume that there exists a unique $i_0$ such that $\tau(i_0)=i_0$, and with $a_{i_0}$ odd.  To show the ``if" part of $(i)$ and part $(ii)$,   it suffices to prove the formula of cardinality and that the strata $\CV^o_W$ have the desired type.

The decomposition (\ref{prodirr})  induces a decomposition as a direct sum of generalized eigenspaces
$$
V=\bigoplus_{i=1}^\ell V_i,\quad V_i:={\rm Ker}\, P_i^{a_i}(\bar g).
$$

 For $v_i\in V_i,v_j\in V_j$, there is some non-zero constant $c$ such that (cf. \cite{AG})
$$
0=\langle P_i^{a_i}(\bar g)v_i,v_j \rangle=c\langle v_i, P^{\ast a_i}_i(\bar g^{-1}) v_j \rangle=c\langle v_i, P^{\ast a_i}_{i}(\bar g) \bar g^{-s_i}v_j \rangle ,
$$
where $s_i=a_i{\rm deg}\, P_i$. 
Then we have two cases:
\begin{itemize}\item  If $\tau(i)=i$, by the above equation we see that $V_{i}$ is orthogonal to $\oplus_{j\neq i}V_j$ and the restriction of the Hermitian form to $V_{i}$ is non-degenerate.
\item If $\tau(i)=j\neq i$, then $V_i\oplus V_j$ is orthogonal to $V_k,k\neq i,j$. And the restriction of the hermitian form to $V_i\oplus V_j$ is non-degenerate, both $V_i$ and $V_j$ being totally isotropic subspaces.
\end{itemize}

Consider the decomposition 
$$W=\bigoplus_{i}W_i,\quad W_i:=W\cap V_i.
$$
 Then each $W_i$ is invariant under $\bar g$ and totally isotropic in each $V_i$. 
By the regularity of $\bar g$, we may list all $\bar g$-invariant subspaces in $V_{i}$: for each $m=0,1,\ldots,a_{i}$:  there is precisely one invariant subspace (denoted by $V_{i,m}$) of dimension $m\cdot {\rm deg}\, P_{i}$ and these exhaust all invariant subspaces of $V_i$. Moreover $V_{i,m}={\rm Ker}\, P_i^m(\bar g)$.  Let now $i=i_0$. The proof of  the ``only if" part of $(i)$ shows that $W_{i_0}=V_{i_0,\frac{a_{i_0}-1}{2}}$.  We also know that $W':=\bigoplus_{i\neq i_0}W_i$ must be maximal totally isotropic in $V':=\bigoplus_{i\neq i_0}W_i$. Now suppose that $i\neq i_0$.  We have two cases.

\begin{itemize}\item  If $\tau(i)=i$,  $W_{i}$ must be a maximal totally isotropic subspace of $V_i$.  Hence $W_i=V_{i,a_{i}/2}$ is unique (note that $a_i$ must be even by the assumption of the uniqueness of $i_0$).  
\item If $\tau(i)=j\neq i$, then $W_i\oplus W_j$ must be a maximal totally isotropic subspace of $V_i\oplus V_j$. Therefore $W_j$ is uniquely determined by $W_i$ and we can take $W_i=V_{i,m}$ for $m=0,1,\ldots, a_i$. We thus have precisely $a_{i}+1=a_{j}+1$ number of choices.
\end{itemize}

In summary we have shown that the cardinality of the set of $W$ with $\CV_W^{o,\bar g}$ non-empty is
$$
\prod_{\{i,j\},j=\tau(i)\neq i}(1+ a_i)\cdot \prod_{i\neq i_0,\tau(i)=i}1.
$$

Moreover, this also shows that the type of all such $\CV_W^{o}$ is the same, namely ${\rm deg}\, P_{i_0}$. And by Lemma \ref{cardfix}, for each $W$, the cardinality of $\CV_W^{o,\bar g}$ is precisely ${\rm dim}\,  W^\perp/W={\rm deg}\, P_{i_0}$.
We conclude that the cardinality of $(\delta(\CM)\cap \CN^g)(\BF)$ is 
$$
\sum_{W\in {\rm Vert}^{\bar g}(V)}|\CV_W^{o,\bar g}|=\prod_{\{i,j\},j=\tau(i)\neq i}(1+ a_i)\cdot {\rm deg}\,P_{i_0}.
$$
\end{proof}

We now calculate the derivative of the orbital integral using Lemma \ref{lem Cn'}.
\begin{prop}\label{derivform}
 Let $g$ be as above. Let $\gamma\in S_n(F)$ match $g$. Then  $O'(\gamma,1_{S_n(\CO_F)})=0$ unless  there is a unique $i_0$ such that $\tau(i_0)=i_0$ and with $a_{i_0}$ odd,  in which case 

$$
O'(\gamma,1_{S_n(\CO_F)})=-\omega(\gamma) {\rm log}\,q\prod_{\{i,j\},j=\tau(i)\neq i} (1+a_i)\cdot  {\rm deg}\, P_{i_0}\cdot \frac{a_{i_0}+1}{2}.
$$
\end{prop}

\begin{proof}
By mapping $\Lambda$ to $\Lambda/L$, the set of lattices
$$\{\Lambda\mid L\subset \Lambda\subset L^*, g\Lambda=\Lambda,\Lambda^\tau=\Lambda\}
$$
is in bijective correspondence with  the set of subspaces
$$
\CW:=\{W\mid W\subset V, \bar g W=W,W^{\bar\tau}=W\} ,
$$
where $\bar \tau$ is the involution on $V$ induced by the restriction of the involution $\tau$ to  $L^\ast$.  
We need to describe this involution. By definition, the involution $\tau$ on $C$ has the property that
$$
(x^\tau,y^\tau )=(y,x),\quad (g\tau)^2=1
$$
The induced involution $\bar \tau$ on $V$ inherits the same properties. In particular,  for a polynomial $P\in k'[T]$, we have
$$
P(\bar g) \bar \tau=\bar \tau \bar P(\bar g^{-1}).
$$
In particular, $\bar \tau$ maps $V_{i,m}={\rm Ker}\,  P_{i}^m(\bar g)$ to $V_{\tau(i),m}$ and this is the reason we use the same notation $\tau$ to denote the involution on the index set $\{1,2,\ldots,\ell\}$. And clearly we have $\ell(\Lambda/L)={\rm dim}_{k'} \Lambda/L$.

  We consider the decomposition $V=\bigoplus_i V_i$ and $W=\bigoplus_i W_i$, where $W_i=W\cap V_{i}$. 

First we assume there is a unique $i=i_0$  such that $\tau(i)=i$ and with $a_i$ odd. 
According to  $W_{i_0}=W\cap V_{i_0}$, we write $\CW=\coprod_{m=0}^{a_{i_0}} \CW_m$ as a disjoint union where $\CW_m$ consists of $W\in \CW$ such that $W_{i_0}= V_{i_0,m}$. Since $V_{i_0,m}$ is invariant under the involution $\bar \tau$, it is clear that the map $W\mapsto W\oplus V_{i_0,m}$ defines a bijection between $\CW_0$ and $\CW_m$.
Therefore we may write the sum of Lemma \ref{lem Cn'}
$$
\sum_{W\in\CW}(-1)^{{\rm dim}\, W}{\rm dim}\, W=\sum_{W\in\CW_0}(-1)^{{\rm dim}\, W} \sum_{m=0}^{a_{i_0}} (-1)^{{\rm dim}\,V_{i_0,m}}({\rm dim}\, W+{\rm dim} V_{i_0,m}).
$$
Note that $a_{i_0}\cdot {\rm deg}\, P_{i_0}$ is odd. Therefore the inner sum simplifies to
$$
-\frac{a_{i_0}+1}{2}{\rm deg}\, P_{i_0} ,
$$
which is independent of $W\in \CW_0$.

We now compute the sum $\sum_{W\in\CW_0}(-1)^{{\rm dim}\, W} $. As before we still use the notation  $V'=\bigoplus_{i\neq i_0}V_i.$ 
Note that $\CW_0=\{ W\subset V' \mid W\in \CW\}$.  Then for $W\in\CW_0$ we have the decomposition $W=\bigoplus_{i\neq i_0}W_i$. Similar to the proof of the previous proposition, we have two cases for $i\neq i_0$:
\begin{itemize}\item  If $\tau(i)=i$,  then $a_i$ is even and there are $a_i+1$ choices of $W_{i}=V_{i,m}$ for $m=0,1,\ldots,a_i$.
\item If $\tau(i)=j\neq i$, then $W_j=\bar \tau W_i$ and there are $a_i+1$ choices of $W_{i}=V_{i,m}$ for $m=0,1,\ldots,a_i=a_j$.
\end{itemize}
Hence  $\sum_{W\in\CW_0}(-1)^{{\rm dim}\, W}$ is equal to the product of
$$
\sum_{m=0}^{a_i}(-1)^{2{\rm dim} V_{i,m}} =a_i+1
$$
for each pair $(i,j)$ with $j=\tau(i)\neq i$,  and 
$$
\sum_{m=0}^{a_i}(-1)^{{\rm dim} V_{i,m}}=1
$$
for $i\neq i_0$ with $\tau(i)=i$ (and since $a_i$ is even).
This proves the formula when there is a unique $i$ such that $\tau(i)=i$ and $a_i$ odd. 

Now suppose that there are at least two such $i$'s. We claim that  then $O'(\gamma,1_{S_n(\CO_F)})=0$.  
The same argument as above shows that $\sum_{W\in\CW_0}(-1)^{{\rm dim}\, W}$ is a product which has a factor of the form
$$
\sum_{m=0}^{a_i}(-1)^{{\rm dim} V_{i,m}}=0 ,
$$
when $i\neq i_0$ with $\tau(i)=i$ and $a_i$ odd! This shows that $O'(\gamma,1_{S_n(\CO_F)})=0$ and hence completes the proof.
\end{proof}

\section{Calculation of length}

We continue to assume $\pi\cdot (L^*/L)=(0)$, and keep all other notation from the previous section. In particular, we assume that $(\delta(\CM)\cap \CN^g)(\BF)$ is non-empty. 

A point  of $(\delta(\CM)\cap \CN^g)(\BF)$ corresponds to a lattice $B$  in $C_n\otimes_E{\breve F}$, occurring in a chain of inclusions of
 lattices, 
\begin{equation}
L\otimes_{\CO_E} \CO_{\breve F}\subset \Lambda^*\otimes_{\CO_E} \CO_{\breve F}\subset B^*\subset^1B\subset \Lambda\otimes_{\CO_E} \CO_{\breve F}\subset L^*\otimes_{\CO_E} \CO_{\breve F} ,
\end{equation}
or also to a subspace $U$  of $V\otimes_{k'} \bar k$,  occurring in a chain of inclusions of vector spaces over $\bar k$, all of which are invariant under $\bar g$, 
\begin{equation}\label{chain}
(0)\subset W\otimes_{k'}\bar k\subset U^\perp\subset^1U\subset W^\perp\otimes_{k'}\bar k\subset V\otimes_{k'} \bar k . 
\end{equation}

Let $\lambda$ be the eigenvalue of $\bar g|(U/U^{\perp})$. Then by the regularity of $\bar g|U$, there exists a unique 
Jordan block to $\lambda$ in $U$. The size of this Jordan block is of the form $c+1$, where $c$ is the size of the Jordan block 
of $\bar g|U^\perp$. 
The size of the Jordan block of $\bar g$ to $\lambda$ is equal to the exponent $a_{i_0}$ with which the irreducible polynomial $P=P_{i_0}$ occurs in $P_{\bar g}$ and is equal to $a_{i_0}=2c+1$.
To see this, consider the decomposition $\prod P_i^{a_i}$ of the characteristic polynomial of  $\bar{g}$ into irreducible factors. Over $\bar{k}$ only  $\lambda$ is a zero of $P_{i_0}$. Since $P_{i_0}$ is an irreducible polynomial over a finite field, $P_{i_0}$ has only simple zeros in  $\bar{k}$.  Since the minimal polynomial of $\bar{g}$ equals the  characteristic polynomial of $\bar{g}$ it follows that the size of the (unique) Jordan block of  the eigenvalue $\lambda$ is the multiplicity of the zero $\lambda$ in the characteristic polynomial, and  this is  $a_{i_0}$. Now we use the chain of inclusions \eqref{chain} and the formulas \eqref{decomp} and \eqref{propdecomp} to conclude that   $a_{i_0}=2c+1$. 
\begin{proposition}\label{length} Assume $F=\BQ_p$. 
Suppose that $n\leq2p-2$.
 The length of the local ring of $\delta(\CM)\cap \CN^g$ at the point $[B]\in (\delta(\CM)\cap \CN^g)(\BF)$ corresponding to $B$  is equal to $c+1=\frac{1}{2}(a_{i_0}+1)$. 
\end{proposition} 
We note that this proposition, together with Propositions \ref{cardform} and \ref{derivform}, proves assertion (iv) of Theorem \ref{mainthm}.

We first determine the tangent space of $(\CM\cap\CN^g)\otimes\BF$ at the point $[B]$ corresponding to $B$.
\begin{lemma}\label{tangspace}
The tangent space of $(\CM\cap\CN^g)\otimes\BF$ at $[B]$ is a one-dimensional subspace of the tangent space of $\CM\otimes\BF$ if $c\geq 1$. If $c=0$, the tangent space is trivial. 
\end{lemma}
\begin{proof}
For this, we have to first recall how one associates the lattice $B$ to a point $(X, \iota, \lambda, \rho)$ of $\CN_n(\BF)$. Let 
$M(X)$ be the Dieudonn\'e module of $X$.   Then $\iota$ induces a $\BZ/2$-grading of $M(X)$, and $\lambda$ induces an alternating form $\langle\, , \, \rangle$ on $M(X)$. Furthermore, $\rho$ induces an identification of $M(X)_0\otimes _{\CO_{\breve F}}{\breve F}$ with $C_n\otimes_E{\breve F}$ such that ${\rm id}_{C_n}\otimes \sigma^2$ corresponds to $\pi V^{-2}$, and the extended form on $C_n\otimes_E{\breve F}$ to $x, y\mapsto 
\pi^{-1}\delta^{-1}\langle x , \pi V^{-1} y \rangle$. Here $\delta$ is a fixed element of $\CO_E^\times$ with $\bar \delta=-\delta$. We then have $B=M(X)_0$, and $B^*=\pi V^{-1}M(X)_1$. 

Now the tangent space is given by ${\rm Hom}_{[B]}({\rm Spec}\, \BF[\epsilon], \CN_n)$, where the index $[B]$ indicates that only morphisms are considered whose image point is $[B]$. By Grothendieck-Messing theory, the tangent space $\CT_{\CN_{n}, [B]}$ is equal to 
$$
{\rm Hom}_{[B]}({\rm Spec}\, \BF[\epsilon], \CN_n)={\rm Hom}_{0, {\rm isot}}(VM(X)/\pi M(X), M(X)/VM(X)) ,
$$
where the index indicates that only homomorphisms are considered which respect the $\BZ/2$-grading and the alternating form. Hence we have 
$$
\begin{aligned}
\CT_{\CN_{n}, [B]}=& \ {\rm Hom}(VM(X)_0/\pi M(X)_1, M(X)_1/VM(X)_0)\\ 
\simeq & \ {\rm Hom}(M(X)_0/\pi V^{-1}M(X)_1, V^{-1}M(X)_1/M(X)_0)\\
\simeq&\ {\rm Hom}(B/B^*, B^*/\pi B).
\end{aligned}
 $$
 Similarly, if $[B]\in \CN^g$, the tangent space to $\CN^g$ at $[B]$ is given by ${\rm Hom}_g(B/B^*, B^*/\pi B)$, where the index indicates that only $g$-equivariant homomorphisms are considered. And if $[B]\in \CM$, then 
 $$
 B=A\oplus \CO_{\breve F}u \text { and } B^*=A^*\oplus \CO_{\breve F}u ,
$$ and 
$$
B/B^*=A/A^* \text { and } B^*/\pi B=A^*/\pi A\oplus \BF\bar u .
$$ Then the tangent space to $\delta(\CM)$ at $[B]$ is equal to the subspace of   ${\rm Hom}(B/B^*, B^*/\pi B)$ consisting of homomorphisms   whose image is contained in  $A^*/\pi A$, and the 
tangent space to $\delta(\CM)\cap\CN^g$ at $[B]$ is given by the subspace  of ${\rm Hom}_g(B/B^*, B^*/\pi B)$ of elements which factor through $A^*/\pi A$.  In other words, this tangent space is identified with the intersection of the eigenspace $(B^*/\pi B)(\lambda)$ to $\lambda$ in $B^*/\pi B$ with $A^*/\pi A$. 

 We now show that this intersection, denoted by $\CT$, has dimension one if $c\geq 1$ and zero if $c=0$.  Since $(B,B^\ast)\subset \CO_{\breve F}$ and $B^\ast\subset B$, we have an induced sesqui-linear pairing on $B^\ast/\pi B$ valued in $\BF$. We still denote this pairing by $(\cdot,\cdot)$.
As $u\in B^\ast$, we may denote by $\bar u$ its image in the $\BF$-vector space $B^\ast/\pi B$. The map $\bar b \mapsto (\bar b,\bar u)$ defines an $\BF$-linear functional denoted by $\ell_u$ on $B^\ast/\pi B$. Using this, we may identify $A^\ast/\pi A$ with the kernel of $\ell_u$. If $\bar b\in \CT$,  it is an eigenvector of $\bar g$ with eigenvalue $\lambda$, and $(\bar b, \bar u)=0$. Since $\bar b$ is eigenvector of $\bar g$ we have
$$
\bar g^{-1}\bar b=\lambda^{-1}\bar b ,
$$
and hence
$$
(\bar g^{-1}\bar b,u)=\lambda^{-1}(\bar b,\bar u)=0 ,\,\text{ i.e., }\, (\bar b,\bar g \bar u)=0 .
$$
Similarly, $(\bar b,\bar g^i\bar u)=0$ for all $i\in \BZ$. This implies that $(b,L)=0 \mod \pi$, where $b\in B^\ast $ is any lifting of $\bar b\in B^\ast/\pi B$. Equivalently we have $(b/\pi,L)\in \CO_{\breve F}$, and hence $b\in \pi L^\ast\otimes_{\CO_E} \CO_{\breve F}$. We note the following sequence of inclusions,
\begin{equation}
\pi B\subset \pi L^\ast\otimes_{\CO_E} \CO_{\breve F}\subset L\otimes_{\CO_E} \CO_{\breve F}\subset \Lambda^*\otimes_{\CO_E} \CO_{\breve F}\subset B^*\subset^1B .
\end{equation}
We have proved that $\CT$ is a subspace of $X:=(\pi L^\ast\otimes_{\CO_E} \CO_{\breve F})/\pi B$. And, in fact,  $\CT$ is precisely the $\lambda$-eigenspace in $X$. Now $X$ is obviously isomorphic to $Y:=(L^\ast\otimes_{\CO_E} \CO_{\breve F})/ B$ as $\bar g$-modules, and hence is a $\bar g$-cyclic $\BF$-vector space. It is easy to see that the $\lambda$-eigenspace $Y(\lambda)$ of $Y$ is one-dimensional when $c\geq 1$ and zero if $c=0$.
\end{proof}

It follows from the preceding lemma that the  completed local ring $R=\widehat{\CO}_{\CM\cap \CN^g, [B]}$ is an  $\CO_{\breve F}$-algebra of the form 
\begin{equation}
R=\CO_{\breve F}[[t]]/I  \text{\, if $c\geq 1$, \, resp. $R=\CO_{\breve F}/I$\, 
 if $c=0$},
\end{equation}
where $I$ is an ideal in $\CO_{\breve F}[[t]]$, resp.\ in $\CO_{\breve F}$. Therefore Proposition \ref{length} follows from the following proposition.
\begin{proposition}\label{ideal} Assume $F=\BQ_p$. 
Suppose that $n\leq 2p-2$. Then
$I=(\pi, t^{c+1}) . $
\end{proposition}
The fact that $\pi\in I$ follows  from the relation to the special divisors
of \cite{KR}. Recall that to any non-zero element $x\in C_n$, there is associated the special divisor $\CZ(x)$ of $\CN_n$, cf.\ \cite{KR}, Lemma 3.9. It is  the closed formal subscheme of $\CN$ with  $S$-valued points
\begin{equation*}
\begin{aligned}
 \{(X,\iota, \lambda, \rho)\mid \text{ the composed quasi-homomorphism } &\overline{Y}\times_{\CO_{\breve F}}\bar S\stackrel{x}{\longrightarrow}{\BX}\times_{\BF}
\bar{S}\stackrel{\varrho_X^{-1}}{\longrightarrow}X\times_S \bar{S}
\\
 \text{ lifts to an $\CO_E$-linear homomorphism }&\overline{Y}\times_{\CO_{\breve F}}S\to X\}. 
 \end{aligned}
\end{equation*}
Here we have identified $C_n$ with $\Hom_{\CO_E}(\overline Y \times_{\CO_{\breve F}}\BF, \BX)\otimes_{\BZ}\BQ$ as explained in \cite{KR}, Lemma 3.9. The elements of $\Hom_{\CO_E}(\overline Y \times_{\CO_{\breve F}}\BF, \BX)\otimes_{\BZ}\BQ$ are called {\em special homomorphisms},  cf.\ \cite{KR}.
The special divisor $\CZ(x)$ is a relative divisor, 
with set of $\BF$-points equal to 
\begin{equation}
\CZ(x)(\BF)=\{ B\in \CN_n(\BF)\mid x\in B^* \} .
\end{equation}
Similarly, if ${\bf x}=[x_1,\ldots, x_m]\in (C_n)^m$, then $\CZ({\bf x})=\CZ(x_1)\cap\ldots\cap \CZ(x_m)$ has $\BF$-points equal to 
\begin{equation*}
\CZ({\bf x})(\BF)=\{ B\in \CN_n(\BF)\mid \{x_1,\ldots, x_m\}\subset B^* \} .
\end{equation*}
Now let $g\in G_n$ be regular semi-simple. Then 
\begin{equation}\label{reltosp}
\delta(\CM)\cap \CN^g\subset \CZ(u, gu, \ldots,g^{n-1}u) .
\end{equation}
Indeed, $\delta (\CM)$ can be identified with $\CZ(u)$,  comp.\  \cite{KR}, Lemma 5.2. Hence the assertion follows by the $g$-invariance of the LHS in \eqref{reltosp}.  

Note that the fundamental matrix of $(u, gu,\ldots, g^{n-1}u)$ in the sense of \cite{KR} is equivalent to the diagonal matrix $\pi^{{\rm inv}(g)}$. Therefore we may apply the following theorem.
\begin{theorem}\label{speciald} Assume $F=\BQ_p$. 
Let ${\bf x}=[x_1,\ldots, x_n]\in (C_n)^n$  with fundamental matrix $T({\bf x})$ equivalent to 
$\pi^\mu$, where $\mu=(1^{(m)}, 0^{(n-m)})$ is minuscule. Then 
$$\CZ(x_1, x_2,\ldots,x_n)\subset \CN_n\otimes_{{\CO_{\breve F}}}\bar k .$$ 
\end{theorem}
The proof is given in \S \ref{Proof1}. Assuming this theorem, we may write
\begin{equation}\label{equalk}
R=\BF[[t]]/\bar I ,
\end{equation}
with an ideal  $\bar I\subset \BF[[t]]$. Note that at this point, the case $c=0$ is proved completely (in particular, in this case the restriction $n\leq 2p-2$ is not needed). The general case follows from the following theorem which together with Theorem \ref{speciald} implies Propositions \ref{ideal} and \ref{length}. 
\begin{theorem}\label{idealk} Assume $F=\BQ_p$. 
If $n\leq 2p-2$, the ideal $\bar I\subset \BF[[t]]$ equals $(t^{c+1})$.
\end{theorem} 
The proof is given in \S \ref{proofidealk}. 
\begin{corollary} Assume $F=\BQ_p$. Let $n\leq 2p-2$. 
Assume that $g$ is regular semisimple and that $L_g^\ast/L_g$ is killed by $\pi$.  Then the contribution of $\Lambda\in {\rm Vert}^g(L_g)$  to the intersection multiplicity, if non-zero,   is equal to
$$
{\rm mult}(\Lambda)={\rm deg} \, P_{i_0} \frac{a_{i_0}+1}{2}.
$$
\end{corollary}
Of course, we are using here the notation of Proposition \ref{cardform}. Now all assertions of Theorem \ref{mainthm} are proved.

\section{Proof of Theorem \ref{speciald}}\label{Proof1}

In this section, we assume $F=\BQ_p$. Accordingly, we write $\BQ_{p^2}$ for $E$, and $W$ for $\CO_{\breve F}$. In the terminology of \cite{KR},  we will prove the following theorem. 
\begin{theorem}\label{specialdQ}
Let $j_1,\ldots,j_n$ be special homomorphisms such that the corresponding fundamental matrix $T(j_1,\ldots,j_n)$ is equivalent to a diagonal matrix of the form ${\rm diag}(p,\ldots,p, 1,\ldots,1)$ (where $p$ occurs say $m$ times and $1$ occurs $n-m$ times). Let 
$\CZ=\bigcap_{i=1,\ldots,n}\CZ(j_i)\subseteq \mathcal N$. Then $\CZ$ is an integral scheme. In particular, $p\cdot \mathcal{O}_{\CZ}=0$. In fact, $\CZ$  is equal to $\CV(\Lambda)$ for some vertex lattice $\Lambda$ in $C_n$ of type $m$. 
\end{theorem}

\begin{remark}\label{DLmod}
We point out that this theorem gives a modular interpretation of the closure $\CV(\Lambda)$ of the Deligne-Lusztig variety $\CV(\Lambda)^o$. 
Here $\Lambda=\langle x_1,\ldots, x_n\rangle^*$ is the dual of the lattice generated by the elements $x_1,\ldots, x_n$ of $C_n$ corresponding to 
$j_1,\ldots,j_n$. 
\end{remark}

 First we remark that we may  replace $n$ by $m$, cf. \cite{KR}, proof of Lemma 5.2.  Hence we may assume that $T(j_1,\ldots,j_n)$ is equivalent to the diagonal matrix  ${\rm diag}(p,\ldots,p)$. 

We use the following simple fact. 

\begin{lemma}\label{pg0}
Let $\CO$ be a complete discrete valuation ring, with uniformizer $\pi$ and algebraically closed residue field $k$. Let 
 $\mathcal Y$ be a  (formal) 
 scheme locally (formally) of finite type over Spf $ \CO$ and such that its special fiber $\mathcal{Y}_k$ is regular. Suppose that there does not exist a $\CO/(\pi^2)$-valued point of $\mathcal Y$. Then $\pi\cdot \mathcal{O}_{\mathcal{Y}}=0$ .
\end{lemma}

\begin{proof} Suppose the claim is false.
Then there exists a $k$-valued point $x$ of $\CY$ such that $\pi\neq 0$ in $\mathcal O_{\mathcal Y, x}$. We show that under this assumption there is an $\CO/(\pi^2)$-valued point of $\CY$ with underlying $k$-valued point $x$. 
 Locally around $x$, the (formal) scheme $\mathcal Y$ is a closed (formal) subscheme of a (formal) scheme $\mathcal X$ which is locally of finite type over $\CO$ and smooth over $\CO$.    Let $R=\widehat{\mathcal O}_{\mathcal X, x}$.  We may identify $R$ with $\CO$ or with $ \CO[\![x_1,\ldots,x_{N}]\!] $ for some $N\geq 1$. In the first  case there is nothing to do, so we assume the second case.   Let $I$ be the ideal of $\mathcal Y$ in $R$. Thus we assume that $\pi\not\in I$. It is enough to construct an $\CO$-linear homomorphism $\psi: R/I\rightarrow \CO/(\pi^2)$. 
 Let $\mathfrak m$ be the maximal ideal of $R$ and let $\mathfrak m'$ be the maximal ideal of $R/(\pi)$. Let $l=\dim \mathcal{X}_k-\dim \mathcal Y_k$. Since $\mathcal Y_k$ is regular, we find $l$ distinct elements $q_{1},\ldots,q_{l}\in I$ such that the images of $q_{1},\ldots,q_{l}$ in $R/(\pi)$ generate the ideal of $\mathcal Y_k$ in $R/(\pi)$ and such that
the images of $q_{1},\ldots,q_{l}$ in $\mathfrak m' /{\mathfrak m'}^ 2$ 
are linearly independent. We extend the $q_i$ to a system of generators $q_1,\ldots,q_r$ of $I$. For $i\leq l$, let $y_i=q_i$. We find elements $y_{l+1},\ldots,y_{N}\in \mathfrak m$ such that the images of $y_1,\ldots,y_{N}$ in $\mathfrak m' /{\mathfrak m'}^ 2$  form a basis of $\mathfrak m' /{\mathfrak m'}^ 2$. 
Thus  $\CO[\![x_1,\ldots,x_{N}]\!] = \CO[\![y_1,\ldots,y_{N}]\!] $. Now we consider the $\CO$-linear homomorphism $\phi:  \CO[\![y_1,\ldots,y_{N}]\!] \rightarrow \CO$ given by $y_i \mapsto \pi^2$  for  all $i$. 

\smallskip

\noindent{\bf Claim }{\em The image of the ideal $I$ under $\phi$ is $(\pi^2)$.}

\smallskip

\noindent The ideal $\phi(I)$ is generated by $\phi(q_1),\ldots,\phi(q_r)$. By definition we have $\phi(q_i)=\pi^2$ for $i\leq l$. Now assume that $i > l$.  Then $q_{i}=\sum_{k\leq l} c_kq_k + \pi\cdot z$ for suitable elements $c_1,\ldots,c_l,z\in R$ (depending on $i$). Since we assume $\pi\not\in I$, it follows that $z$ is not a unit, hence $z\in \mathfrak m = (\pi,y_1,\ldots,y_{N})$. Hence $\phi(z)\in (\pi)$ and  $\phi(q_i)\in (\pi^2)$. This confirms the claim.

Using the claim it follows that $\phi$ induces  a $\CO$-linear homomorphism $\psi: R/I\rightarrow \CO/(\pi^2)$ yielding an $\CO/(\pi^2)$-valued point of $\mathcal Y$. But this contradicts our assumption. Hence $\pi\cdot \mathcal{O}_{\mathcal{Y}}=0$. 
\end{proof}
\begin{remark}
Taking into account Grothendieck's infinitesimal characterization of smoothness, the previous lemma gives a purely infinitesimal \emph{sufficient condition} for a (formal) $\CO$-scheme to be a (formal) $k$-scheme. 
\end{remark}
We will prove Theorem \ref{specialdQ} by showing  in Proposition \ref{liftW} and Corollary \ref{regu} that $\CZ$ satisfies the hypotheses of the previous lemma. 

For any (formal) $W$-scheme $S$,  we denote by $S_p$ its special fiber.
\begin{proposition}\label{liftW}
Let $j_1,\ldots,j_n$ and $\CZ$ be as in the theorem. Then $\CZ$ does not have a $W/(p^2)$-valued point. 
\end{proposition}

\begin{proof} We may assume that $j_1,\ldots,j_n$ all have valuation $1$ and are all perpendicular to each other. We also first assume that $n>1$. Suppose there was a $W/(p^2)$-valued point $\varrho$ of $\CZ$. Let $M$ be the Dieudonn\'e module of the underlying $\BF$-valued point. Let $M_{W/(p^2)}=M\otimes_W W/(p^2)$. We obtain a Hodge filtration $\mathcal{F}\hookrightarrow M_{W/(p^2)}$ corresponding to $\varrho$, and lifting the Hodge filtration of the underlying $\BF$-valued point. From the $\mathbb Z_{p^2}$-action we get a decomposition $\mathcal{F}=\mathcal{F}_0\oplus\mathcal{F}_1$, where $\mathcal{F}_0 $ is free of rank of rank $n-1$ and $\mathcal{F}_1$ is free of rank $1$. Let $x_i=j_i(\overline{1}_0)\in M_0$, where we are using the notation of \cite{KR}. We denote the image of $x_i$ in $M_{W/(p^2)}$ by $\overline{x}_i$. Then it follows that $\overline{x}_i\in \mathcal{F}_0$. Let $\overline{f}_1,\ldots,\overline{f}_{n-1}$ be a basis of $\mathcal{F}_0$. Let $f_i\in M_0$ be a lift of  $\overline{f}_i$, and choose $f_n\in M_0$ such that $f_1,\ldots,f_n$ is a basis of $M_0$.  Let $\widehat{x}_i$ be the image of $x_i$ in the span of $f_1,\ldots,f_{n-1}$ (viewed as a quotient of $M_0$). Then $\widehat{x}_1,\ldots,\widehat{x}_n$ are linearly dependent, i.e.,  $\sum_{i}c_i\widehat{x}_i=0$ for suitable $c_i\in W$,   which are not all  zero. We may assume that the valuation of $c_n$ is minimal among the valuations of the $c_i$. Dividing by $-c_n$ we may therefore assume that $\widehat{x}_n=\sum_{i<n}c_i\widehat{x}_i$. Therefore $x_n=\sum_{i<n}c_i{x_i}+cf_n$ for some $c\in W$. Since the image of $x_n-\sum_{i<n}c_i{x_i}=cf_n$ in $M_{W/(p^2)}$ lies in $\mathcal{F}_0 $, it follows that $c$ is divisible by $p^2$.

Now for any $i \neq n$ we have $0=\{x_n,x_i\}=c_i\{x_i,x_i\}+c\{f_n,x_i\}=c_ip+c\{f_n,x_i\}$. Here $\{ , \}$ is the hermitian form on $C=(M_0\otimes \BQ)^{V^{-1}F}$ as in \S 3. Since $x_i\in M_0^*$, it follows that $\{f_n,x_i\}$ is integral. Since further $c$ is divisible by $p^2$, it follows that $c_i$ is divisible by $p$. It follows that $x_n/p\in M_0^* $ so that $j_n$ has valuation bigger than $1$, a contradiction which completes the proof in the case $n>1$. 

Finally,  we observe that this reasoning also works for $n=1$, since it shows that in this case $x_n=cf_n$, where $c$ is divisible by $p^2$ so that $j_n$ cannot have valuation $1$. 
\end{proof}
\begin{lemma}\label{supergen}
 Let $n$ be odd. Let $x$ be a $\BF$-valued point of $\mathcal{N}$. The following conditions are equivalent.

(i) $x$ lies on only one irreducible component of $\mathcal N_{\rm red}$. 

(ii) No special cycle of valuation $0$ passes through $x$.

(iii) The Dieudonn\'e module modulo $p$ of $x$ is of type $\mathbb B(n)$, in the sense of \cite{VW}, \S 3 (cf. also the beginning of the proof of Theorem \ref{prep} below).
\end{lemma}
We remark that (ii) and (iii) both imply that $n$ is odd. A point satisfying (ii) is called \emph{super-general}.
\begin{proof}
The equivalence of (ii) and (iii) follows from \cite{BW}, Proposition 3.6 and Lemma
4.1.  

Next we prove the equivalence of (i) and (ii). 
 The point $x$ lies on two irreducible components if and only if there are  two
vertex lattices $\Lambda_1$ and $\Lambda_1'$ of type $n$ in $C_n$ such that $x$ is
a $\BF$-valued point of the corresponding irreducible components
$\mathcal{V}(\Lambda_1)$ and $\mathcal{V}(\Lambda_1')$. This is equivalent to the
statement that $x\in \mathcal{V}(\Lambda)(\BF)$ for some vertex lattice $\Lambda$
of type $t<n$. (Given $\Lambda_1$ and $\Lambda_1'$ define    $\Lambda$ as
$\Lambda_1\cap \Lambda_1'$.)
We claim that for any vertex lattice $\Lambda$ of type $t<n$ there is a special
homomorphism $j$ of valuation $0$ with $ \mathcal{V}(\Lambda)\subseteq
\CZ(j)$.
To see this, note that  $\Lambda$ has an orthogonal basis $e_1,\ldots,e_n$ such that
$\{e_i,e_i\}=1/p$ for $i\leq t$ and $\{e_i,e_i\}=1$ for $i> t$. Let  $j$ be the  special
homomorphism with $j(\bar{1}_0)=e_{t+1}$.
Then for the hermitian form $h(\,,\,)$ on the space $\BV$ of special homomorphisms, we have $h(j,j)=1$ and $y=j(\bar{1}_0)\in \Lambda^*$. This shows the claim, which implies $(ii)\!\!\!\implies\!\! \!(i)$.
For the reverse implication, assume that $\Lambda$ is a vertex lattice such that $\Lambda^*$ contains a vector $y$ with $\{y, y\}=1$. Then 
$\Lambda$ cannot be  of type $n$. This shows that $\CV(\Lambda)$ cannot be contained in special divisor of valuation $0$.
\end{proof}

\begin{theorem} \label{prep}
Let $n\geq 3$.
Let $x$ be a super-general $\BF$-valued point of $\mathcal{N}$.
Then the following statements hold.
\smallskip

\noindent (i)
 For any special homomorphism $j$ with $x\in \CZ(j)(\BF) $ and $x\not\in \CZ(j/p)(\BF), $ the special fiber $\CZ(j)_p$ is regular at $x$.

\smallskip
\noindent
(ii) Let $j_1,\ldots,j_n$ be a basis of the $\BZ_{p^2}$-module of special homomorphisms $j$ with $x\in \CZ(j)(\BF)$. Then the intersection $\bigcap\CZ(j_i)_p$ is regular at $x$.
\end{theorem}

\begin{proof} It is enough to show the claims of the theorem in $\widehat{\mathcal{O}}_{\mathcal{N}_p,x}$ instead of $\mathcal{O}_{\mathcal{N}_p,x}$.
Let $(X, \iota, \lambda)$ be the $p$-divisible group with its $\mathbb Z_{p^2}$-action and its $p$-principal polarization corresponding to $x$, and let $M$ be its Dieudonn\'e module. Let $M_p=M_{p,0}\oplus M_{p,1}$ be the reduction mod $p$ of $M.$ Since we assume that no special cycle of valuation $0$ passes through $x$, it follows that $M_p$ is isomorphic to $\mathbb{B}(n)$ and that $n$ is odd. Here we are using the notation of \cite{VW}, \S 3.1. 
This means that we find  bases $\overline{e}_1,\ldots,\overline{e}_n$ of $M_{p,0}$  and  $\overline{f}_1,\ldots,\overline{f}_n$ of $M_{p,1}$ such that $V( \overline{f}_i)=(-1)^i\overline{e}_{i+1}$ for $i<n$, $V(\overline{e}_n)=\overline{f}_1 $, $F(\overline{f}_i)=(-1)^i\overline{e}_{i-1}$ for $i \geq 3$, $F(\overline{f}_2)=-\overline{e}_1$, $F(\overline{e}_1)=\overline{f}_n$
 and for the induced alternating form we have $\langle \overline{e}_i, \overline{f}_j \rangle= \varepsilon_i\delta_{ij}$, where $\varepsilon_i=1$ for $i=1$ and $\varepsilon_i=-1$ for $i>1$. 
 
 We find lifts $e_i\in M_0$ of $ \overline{e}_i$ and  lifts $f_i\in M_1$ of $ \overline{f}_i$ such that still  $\langle{e_i},{f_j} \rangle= \varepsilon_i\delta_{ij}$. 
  Denote by  $T$ the $W$-span of $e_1,f_2,f_3,\ldots,f_n$ and by $L$ the $W$-span of $f_1,e_2,e_3,\ldots,e_n$. Then
\[
M=L\oplus T, \ \ \ VM=L\oplus pT.
\]
Let $h_1=e_1,h_2=f_2,\ldots,h_n=f_n,h_{n+1}=f_1,h_{n+2}=e_2,\ldots,h_{2n}=e_n.$ 
Define the matrix $(\alpha_{ij})$ by 
\[
Fh_j=\sum_i \alpha_{ij}h_i \text{ for } j=1,\ldots,n,  \\
 \]
 \[
 V^{-1}h_j=\sum_i \alpha_{ij}h_i \text{ for } j=n+1,\ldots,2n.
\]
Since we know the action of $F$ resp. $V$ on the $\overline{e}_i$ and $\overline{f}_i$, we can conclude that $V^{-1}(f_1)=e_n + \sum_{i<n}x_ie_i+pe$ for suitable $x_i\in W$ and $e\in M_0$. Similarly  for $i\geq 2$ we have $ V^{-1}(e_i)=(-1)^{i-1}f_{i-1}+y_if_n+pg_i$ for suitable $y_i\in W$ and $g_i\in M_1$.

\noindent Thus $(\alpha_{ij})$ is of the form
\[
(\alpha_{ij})=
\left(
\begin{array}{cccccc|ccccccccccccccccc}
&-1&&&&&x_1&&&&\\
&&&&&&&&1&&& \\
&&&&&&&&&-1&&\\
&&&&&&&&&&\ddots  \\
&&&&&&&&&&&-1 \\
&&&&&&&&&&&&1 \\
1&&&&&&&y_2&y_3&y_4&\ldots&y_{n-1}&y_n\\
\hline
&&&&&&&-1&&&&\\
&&-1&&&&x_2&&&&\\
&&&1&&&x_3&&&&\\
&&&&\ddots&&\vdots&&&&\\
&&&&&-1&x_{n-1}&&&&&\\
&&&&&&1&&&&&\\
\end{array}
\right) +pD,
\] 
where $D$ has entries in $W$ and also maps $M_0$ to $M_1$ and $M_1$ to $M_0$ .
(The  vertical and horizontal lines    divide the first matrix into four $n\times n$ matrices, and only non-zero entries are displayed.)
It follows (see \cite{Zi}, p. 48) that the universal deformation of $X$ over $\BF[\![ t_{11},\ldots,t_{nn}]\!]$ corresponds to the display $(L\oplus T)\otimes W(\BF[\![t_{11},\ldots,t_{nn} ]\!]) $ with matrix $(\alpha_{ij})^{\rm univ}$ (wrt. the basis $h_1,\ldots, h_{2n}$ and with entries in $W(\BF[\![t_{11},\ldots,t_{nn} ]\!])$ given by 
\[
(\alpha_{ij})^{\rm univ}=
\begin{pmatrix}
1&&&[t_{11}]&\hdots&[t_{1n}] \\
&\ddots&&\vdots&\ddots&\vdots \\
&&1&[t_{n1}]&\hdots&[t_{nn}] \\
&&&1&& \\
&&&&\ddots& \\
&&&&&1 \\
\end{pmatrix} \cdot
(\alpha_{ij}).
\]
Here  $[t]$ denotes the Teichm\"uller representative of $t$. 
  Now let $A^{'}=W[\![t_{11},\ldots,t_{nn}]\!]$ and let  $R^{'}=\BF[\![t_{11},\ldots,t_{nn}]\!]$. 
  We extend the Frobenius  $\sigma $ on $W$ to $A^{'}$ by 
  setting $\sigma(t_{ij})=t_{ij}^p.$ 
 Let $R$ be the completed universal deformation ring (in the special fiber) of $X$, together with its $\BZ_{p^2}$-action and its $p$-principal polarization. Then $R$ is a quotient of $R^{'}$ by an ideal $J$. Using the fact that $(\alpha_{ij})^{\rm univ}$ has to respect the $\BZ/2$-grading,  it is easy to see that the ideal describing the deformation of the $\BZ_{p^2}$-action is $(t_{11},t_{ij})_{i,j\neq 1}$. Using this,  it is easy to see that 
 $J=\big ((t_{11},t_{ij})_{i,j\neq 1}, (t_{1i}-t_{i1})_{i\leq n}\big)$. (Compare also \cite{G}, p.~231.) Thus we may identify $R$ with the ring $\BF[\![t_{2},\ldots,t_{n}]\!]$, where $t_i$ corresponds to the image of $t_{1i}$ in  $R^{'}/J$. We also define 
 $A=W[\![t_{2},\ldots,t_{n}]\!]$.
 For any $m \in \BN$,  denote by $\mathfrak{a}_m$ resp. $\mathfrak{r}_m$ the ideal in $A$ resp. in $R$ generated by the monomials $t_2^{a_2}\cdot \ldots \cdot  t_n^{a_n}$, where  $a_i\geq 0$ and $\sum a_i = m$. Hence $\mathfrak{r}_m=\mathfrak m^m, $ where $\mathfrak m$ denotes the maximal ideal of $R$. Let $A_m=A/\mathfrak{a}_m$ and $R_m=R/\mathfrak{r}_m$.  Then $A^{'}$ is a frame for $R^{'}$, resp. $A$ is a frame for $R$, resp. $A_m$ is a frame for $R_m$. (See \cite{Zi2} for the definition of frames.)

For an  $A^{'}$-$R^{'}$-window $(M^{'}, M^{'}_1, \Phi^{'},\Phi^{'}_1)$, let  $M_1^{'^{\sigma}}=A^{'}\otimes_{A^{'}, \sigma}M_1^{'}$ and denote by $\Psi^{'}:  M_1^{'^{\sigma}} \rightarrow M^{'}$ the linearization of $\Phi^{'}_1$. It is an isomorphism of  $A^{'}$-modules.
 Denote by  $\alpha^{'}:M_1^{'} \rightarrow M_1^{'^{\sigma}}$ the composition of the inclusion map $M_1^{'}  \hookrightarrow M^{'} $ followed by $\Psi^{'^{-1}}$. In this way, the category of formal $p$-divisible groups over $R^{'}$ becomes equivalent to the category of pairs $(M_1^{'}, \alpha^{'})$ consisting of a free $A^{'}$-module of finite rank and an $A^{'}$-linear injective homomorphism $\alpha^{'}:M_1^{'} \rightarrow M_1^{'^{\sigma}}$ such that Coker  $\alpha^{'}$ is a free $R^{'}$-module, and satisfying the {\it nilpotence condition} \cite{Zi}. Since we will only consider deformations of formal $p$-divisible groups,  the 
 nilpotence condition will be  fulfilled automatically, and we will ignore it, comp. also \cite{KR}, section 8. A corresponding description holds for the category of  formal $p$-divisible groups over $R$ resp. $R_m$. 
 
 In the sequel we are using  notation that is customary in Zink's theory. The notation $M_1$ conflicts with its usage when taking the degree-$1$-component of $M$ under the $\BZ/2$-grading. Henceforward we will write $M^1$ for the degree-$1$-component.
 
Let $(\beta_{ij})^{\rm univ}$ be the matrix over $A^{'}$ which is obtained from $(\alpha_{ij})^{\rm univ}$ by replacing the $[t_i]$ by $t_i$ and by multiplying the last $n$ rows by $p$. 
 We consider the $A^{'}$-$R^{'}$-window $(M^{'},M^{'}_1, \Phi^{'},\Phi^{'}_1)$ given by 
$
M^{'}=M\otimes A^{'},  \ M_1^{'}=VM\otimes A^{'}, \  
\Phi^{'}=(\beta_{ij})^{\rm univ}\sigma, \  \  \Phi^{'}_1=\frac{1}{p}\cdot \Phi^{'},
$
where 
the matrix of $\Phi^{'}$ is described in the basis $h_1,\ldots,h_{2n}$. The corresponding display is the universal display described above (easy to see using the procedure described on p.2 of \cite{Zi2}). Hence $(M^{'},M^{'}_1, \Phi^{'},\Phi^{'}_1)$ is the universal window.
Using this and the form of the ideal $J$ given above, one checks that 
 the map  $\alpha: M_1 \rightarrow M_1^{\sigma}$ corresponding to the $A$-$R$ window of the universal defomation of $(X, \iota, \lambda)$  (which is the base change of  $(M_1^{'}, \alpha^{'})$) can be written as follows (using the bases $pe_1,pf_2,..,pf_n,$ $ f_1,e_2,\ldots, e_n$ resp. $p(1\otimes e_1),p(1\otimes f_2),\ldots,p(1\otimes f_n),1\otimes  f_1,  1\otimes e_2,\ldots,1\otimes e_n$ )
  \[
\tilde \alpha=
\left(
\begin{array}{cccccc|ccccccccccccccccc}
&-py_3&py_4&\hdots&-py_n&p&-t_n+y_2+\sum_{i\geq 3}(-1)^{i-1}y_it_{i-1}&&&&\\
-p&&&&&&&t_2&t_3&\hdots&t_{n-1}&x_1+t_n \\
&&&&&&&-1&&&&x_2\\
&&&&&&&&1&&&x_3  \\
&&&&&&&&&\ddots&&\vdots \\
&&&&&&&&&&-1&x_n \\
\hline
&&&&&&&&&&&1 \\
&&&&&&-1&&&&\\
&p&&&&&-t_2&&&&\\
&&-p&&&&t_3&&&&\\
&&&\ddots&&&\vdots&&&&\\
&&&&p&&-t_{n-1}&&&&&\\

\end{array}
\right) +B.
\] 
Here $B=\begin{pmatrix}
B_{11}&B_{12}\\
B_{21}&B_{22}\\ 
\end{pmatrix}
$, where $B_{ij}$ has size $n\times n$, and $B_{11}$ and $ B_{21}$  have entries in $p^2A$ and $B_{12}$ and $B_{22}$ have entries in $pA$. 
Rewriting this in the bases $pe_1,e_2,..,e_n,$ $ f_1,pf_2,\ldots, pf_n$ resp. $p(1\otimes e_1),1\otimes e_2,\ldots,1\otimes e_n,1\otimes  f_1,  p(1\otimes f_2),\ldots,p(1\otimes f_n)$  we obtain a block matrix 
$
\alpha=\begin{pmatrix}
&U\\
\tilde{U}
\end{pmatrix},
$
where $U$ is of the form
 \[
U=
\left(
\begin{array}{ccccccccccccccc}
-t_n+y_2+\sum_{i\geq 3}(-1)^{i-1}y_it_{i-1}&-py_3&py_4&\hdots&-py_n&p\\
-1&&&&\\
-t_2&p&&&\\
t_3&&-p&&\\
\vdots&&&\ddots&\\
-t_{n-1}&&&&p\\

\end{array}
\right) +pB_U,
\] 
where $B_U$ has entries in $A$  and, in the last $n-1$ rows,  even has entries in $pA$, and $\tilde{U}$ is of the form 
 \[
\tilde{U}=
\left(
\begin{array}{ccccccccccccccc}
&&&&&1\\
-p&t_2&t_3&\hdots&t_{n-1}&x_1+t_n\\
&-1&&&&x_2\\
&&1&&&-x_3\\
&&&\ddots&&\vdots\\
&&&&-1&x_{n-1}\\

\end{array}
\right) +pB_{\tilde{U}},
\] where $B_{\tilde{U}}$ has entries in $A$ and in the first row even has entries in $pA$.

The corresponding universal $p$-divisible groups over 
  $R_m$ correspond to the pairs 
  $(M_1(m), \alpha(m))$ obtained by base change from $(M_1, \alpha)$.

Consider the $p$-divisible group $\overline{\mathbb{Y}}$ with its Dieudonn\'e module  $\overline{\mathbb{M}}=W\overline{1}_0\oplus W\overline{1}_1.$ Let $\overline{M}_1=W\overline{1}_0\oplus  W p\overline{1}_1$ and let $n_0=\overline{1}_0$ and  $n_1=p\overline{1}_1$. Then  $\overline{\mathbb{Y}}$ corresponds to the pair $(\overline{M}_{1}, \beta)$ where $\beta(n_0)=1\otimes n_1$ and $\beta(n_1)=-p\otimes n_0$. By base change $W \rightarrow A$ resp.  $W \rightarrow A_m$ we obtain pairs $({\overline{M}_{1}}_{R}, \beta)$ resp. $({\overline{M}_{1}}_{R_m}, \beta)$ corresponding   to the constant $p$-divisible group  $\overline{\mathbb{Y}}$ over $R$ resp. $R_m$. We denote the matrix of $\beta$  by $S$, hence
$$S=
\begin{pmatrix}
0&-p\\
1&0\\
\end{pmatrix}.
$$

Now let $j$ be as in the statement of the theorem, i.e. $x\in \CZ(j)(\BF)$ but  $x\notin \CZ(j/p)(\BF)$. We want to investigate the ideal in $R$ describing the maximal deformation of the homomorphism $j$, and its image in   $R_m$. We will determine explicitly the image of this ideal  in  $R_p.$

The map $j$ corresponds to a map $ j(1):\overline{M}_{1} \rightarrow  M_1(1) $ such that the following diagram commutes,
\[
\xymatrix{ \overline{M}_{1} \ar[d]_{j(1)} \ar[r]^{\beta}&  {\overline{M}}_{1}^{\sigma} 
\ar[d]^{\sigma(j(1))}  \\M_1(1)\ar[r]_{\alpha(1)} & M_1(1)^{\sigma}  .}
\]
Then $j$ lifts over $R_m$ if and only if there is a lift $j(m)$ of $j(1)$ such that the following diagram commutes, 

\[
\xymatrix{ {\overline{M}_{1}}_{R_m} \ar[d]_{j(m)} \ar[r]^{\beta}&  {{\overline{M}}^{\sigma}_{1 {R_m}}}
\ar[d]^{\sigma(j(m))}  \\M_1(m)\ar[r]_{\alpha(m)} & M_1(m)^{\sigma}  .}
\]

We write $j(\overline{1}_0)=a_1 \cdot pe_1+a_2 \cdot e_2+\ldots+a_n\cdot e_n.$ We also write $j(1)=(X(1), Y(1)).$ Then $X(1)$ can be written in the above basis as 
$$ X(1)=
\begin{pmatrix}
a_1&0 \\
\vdots&\vdots\\
a_n&0\\
\end{pmatrix}.
$$
Similarly we write $j(p\overline{1}_1)=b_1 \cdot f_1+b_2 \cdot pf_2+\ldots+b_n\cdot pf_n$ and  $$ Y(1)=
\begin{pmatrix}
0&b_1 \\
\vdots&\vdots\\
0&b_n\\
\end{pmatrix}.
$$

Since $j$ commutes with the Frobenius operator, we have $j(p\overline{1}_1)=FjF^{-1}(p\overline{1}_1)=Fj(\overline{1}_0)$. Using the  matrix $(\alpha_{ij})$ we see that 
$$b_1=-pa_2^{\sigma}+p^2s_1,\, b_i=(-1)^{i}a_{i+1}^{\sigma}+ps_i \text{ for } 2\leq i\leq n-1,\text{ and } b_n=a_1^{\sigma}+\sum\nolimits_{i\geq 2}a_i^{\sigma}y_i+ps_n,$$
 for suitable elements  $s_i\in W.$

Similarly, exploiting the relation  $j(\overline{1}_0)=FjF^{-1}(\overline{1}_0)=-\frac{1}{p}Fj(p\overline{1}_1)$, we obtain the system of equations 
$$a_1=-b_1^{\sigma}x_1/p-b_2^{\sigma}/p+r_1, a_i=-b_1^{\sigma}x_i+(-1)^{i}b_{i+1}^{\sigma}+pr_i \text{ for } 2\leq i\leq n-1,\text{ and }a_n=b_1^{\sigma}+pr_n,$$ 
for suitable elements $r_i\in W$. The equations $b_1=-pa_2^{\sigma}+p^2s_1$ and  $a_1=-b_1^{\sigma}x_1/p-b_2^{\sigma}/p+r_1$ show that $b_1$ and $b_2$ are divisible by $p$. However, not  all $b_i$ are divisible by $p$. Indeed, if they were,  then, because of $b_i=(-1)^{i}a_{i+1}^{\sigma}+pr_i$ for $2\leq i\leq n-1$, the elements $a_3,\ldots,a_n$ would  also be divisible by $p$. Using $a_2=-b_1^{\sigma}x_2+b_{3}^{\sigma}+pr_2$, we see that then also $a_2$ would be divisible by $p$ and, using $b_n=a_1^{\sigma}+\sum_{i\geq 2}a_i^{\sigma}y_i+ps_n$,  it finally would follow that  also $a_1$ is divisible by $p$. However, this would contradict our assumption  that $x\in \CZ(j)(\BF)\setminus \CZ(j/p)(\BF)$.  

We are looking for liftings $X(m)$ of $X(1)$ and  $Y(m)$ of $Y(1)$ over $A_m$ such that
 \begin{equation}\label{displayrec}
 UY(m)= \sigma (X(m))S \  \text{ and } \   \tilde{U}X(m)= \sigma (Y(m))S .
 \end{equation}
Suppose $m=p^l$, where $l\geq1$, and suppose we have found liftings  $X(p^{l-1})$ and  $Y(p^{l-1})$  satisfying \eqref{displayrec}. For any choice of liftings   $X(p^{l})$ and  $Y(p^{l})$  of $X(p^{l-1})$ and  $Y(p^{l-1})$, 
the matrices $\sigma(X(p^{l})),$ resp. $\sigma(Y(p^{l}))$ are equal  to $\sigma(X(p^{l-1}))$ resp. $\sigma(Y(p^{l-1}))$, interpreted as matrices over $A_{p^l}$. Hence there are liftings  $X(p^{l})$ and  $Y(p^{l})$  satisfying \eqref{displayrec} if and only if the matrices 
\[ U^{-1}\sigma (X(p^{l-1}))S \  \text{ and } \  \tilde{U}^{-1} \sigma (Y(p^{l-1}))S\] are integral, and in this case 
\[ X(p^l)= \tilde{U}^{-1} \sigma (Y(p^{l-1}))S  \  \text{ and } \ Y(p^l)= U^{-1}\sigma (X(p^{l-1}))S.\] 
Define now inductively matrices $X_{\BQ}(p^l)$ and $Y_{\BQ}(p^l)$ over $A_{p^l}\otimes_{\BZ}\BQ$ as follows: $X_{\BQ}(1)=X(1)$ and $Y_{\BQ}(1)=Y(1)$ and 
\[ X_{\BQ}(p^{l+1})=   \tilde{U}^{-1} \sigma (Y(p^{l}))S  \  \text{ and } \  Y_{\BQ}(p^{l+1})=U^{-1}\sigma (X(p^{l}))S.\] (Again $\sigma (X_{\BQ}(p^l))$ and $\sigma (Y_{\BQ}(p^l))$ are well defined over $A_{p^{l+1}}\otimes_{\BZ}\BQ$.)
It is easy to see that $Y_{\BQ}(p)$ is integral. 

Let  $\overline{b}_i $ denote the image of $b_i$ in $\BF$. Let $D = \{i\geq 3\mid \overline{b}_i\neq 0\}$. (This set is not empty as we saw above.) Using the form of the matrix $\tilde{U}$ we easily see that $X_{\BQ}(p)$ is of the form 
$$
X_{\BQ}(p)=
\begin{pmatrix}
\frac{1}{p}\sum_{{i}\in D} (-1)^ib_{i}^{\sigma}t_{i-1} & 0\\
0&0\\
\vdots&\vdots\\
0&0\\
\end{pmatrix}+A(p),
$$
where $A(p)$ is integral.  We claim that the equation of $\CZ(j)_p$ in $R/\mathfrak m^p$ is  $\sum_{i\in D} (-1)^i\overline{b}_{i}^{\sigma}t_{i-1} = 0$. 
Let $t=\sum_{i\in D} (-1)^i\overline{b}_{i}^{\sigma}t_{i-1}$ and lift $t$ to an element $\tilde t\in A$ using Teichm\"uller lifts of the coefficients of $t$. Then we claim that $\tilde A:=A/((t_2,\ldots,t_n)^p+\tilde t\cdot A)$ is a frame for  $R/(\mathfrak m^p+t\cdot R)$. Since $\tilde A$ is isomorphic to $W[\![X_1,\ldots,X_{n-2} ]\!]/(X_1,\ldots,X_{n-2})^p$, it is torsion free as an abelian group. Let $\sigma $ be the endomorphism on $\tilde A$ which extends the Frobenius on $W$ by sending (the images of) the $t_i$  to $0$. Then $\sigma$ induces the Frobenius on $R/(\mathfrak m^p+t\cdot R)$. The ideal $p\cdot \tilde A$ is in an obvious way equipped with a pd-structure.
Using this frame, the same calculation as above shows that $j$ lifts over $R/(\mathfrak m^p+t\cdot R)$. Since  $\CZ(j)_p$ is a divisor (\cite{KR}, Proposition 3.5), it follows that we can write the equation for  $\CZ(j)_p$ in  $R/\mathfrak m^p$ in the form $t\cdot s=0$. We have to show that $s$ is a unit. Assume $s$ is not a unit. Then it follows that $j$ lifts over  $R/\mathfrak m^2$. For this ring we have the obvious frame $A/(t_2,\ldots,t_n)^2$.
Again the same calculation as above shows that $j$ does not lift over   $R/\mathfrak m^2$ because $\sum_{{i}\in D} (-1)^ib_{i}^{\sigma}t_{i-1}$ is not divisible by $p$ in  $A/((t_2,\ldots,t_n)^2$.
Thus the equation of $\CZ(j)_p$ in $R/\mathfrak m^p$ is indeed $\sum_{i\in D} (-1)^i\overline{b}_{i}^{\sigma}t_{i-1} = 0$. 
 Claim i) of the theorem follows.

Now we come to claim ii). Let $\mathfrak n$ be the maximal ideal in $\widehat{\mathcal O}_{\CZ, x}$. We need to show ${\rm dim}\,\mathfrak n/\mathfrak n^2={\rm dim}\, \CZ_p$. 
Since $x$ is super-general, it lies on a unique irreducible component of $\mathcal N_{\rm red}$, of the form $\mathcal V(\Lambda)$, where $\Lambda$ is a vertex lattice of type $n$, cf. \cite{KR}, \S 4. Furthermore, $x\in \CZ(j)(\BF)$ if and only if $j(\bar 1_0)\in \Lambda^*=p\Lambda$. 
By Lemma \ref{supergen}, $\ord_p (h(j_i,j_i))\geq 1$ for all $i$, and by the results of \cite{KR}, \S4, the dimension of $\CZ(j_i)_{\rm red}$ is $(n-1)/2$ at $x$ for all $i$. Hence $\CZ_{\rm red}=\mathcal V(\Lambda)$ locally at $x$, and has dimension $(n-1)/2$. We will show that ${\rm dim}\,\mathfrak n/\mathfrak n^2=(n-1)/2$, which will prove that $\CZ_p=\CZ_{\rm red}$ at $x$, and will finish the proof.

We saw above that the equation of $\CZ(j)$ in $\mathfrak m/\mathfrak m^2$ is a linear equation of the form 
$$\sum_{i\geq 1} (-1)^i\overline{b}_i(j)^{\sigma}t_{i-1} = 0,$$
where the coefficients $ \overline{b}_i(j)$ arise by expressing $j(p\bar 1_1)$ in terms of a specific basis of the $\BF$-vector space $VM^0/pVM^0$ with $\overline{b}_1(j)=\overline{b}_2(j)=0$.  We have to see that the rank of this system of linear equations, as $j(\bar 1_0)$ varies through $p\Lambda$, is equal to $(n-1)/2$.  

However, as $j$ varies, the elements $j(p\bar 1_1)$ generate the $W$-lattice $pV(\Lambda\otimes_{\BZ_{p^2}} W)$ inside $VM^0$, and the dimension of $pV(\Lambda\otimes_{\BZ_{p^2}} W)/p VM^0$ is equal to $(n-1)/2$. 
\end{proof}

\begin{corollary}\label{regu}
Let $j_1,\ldots,j_n$  and $\CZ$ be as in Theorem \ref{specialdQ}. Then the special fiber $\CZ_p$ of $\CZ$ is regular.
\end{corollary}
\begin{proof} We use induction on $m$ (notation as in Theorem \ref{specialdQ}). We observe that $m$ is always odd. For $m=1$ there is nothing to do, since for $n=1$ we have $\mathcal{N}_p\cong \Spec\, \BF$. 
If $x$ is super-general, the assertion follows from Theorem \ref{prep}. 

Now assume that  $x$ is not super-general. By Lemma \ref{supergen},  there is a special cycle  $\CZ(j_0)$ of valuation $0$ passing through $x$. We consider the $\BZ_{p^2}$-submodule $J$ of $\mathbb V$ generated by $j_0,j_1,\ldots,j_n$ and we define $\CZ(J)=\cap_{i=0,\ldots,n}\CZ(j_i)$. (Recall that $\mathbb V $ is the $\BQ_{p^2}$-space of special homomorphisms, with hermitian  form $h(\,,\,)$.)

\smallskip

\noindent{\bf Claim}
 {\em There is an orthogonal $\BZ_{p^2}$-basis $b_1,\ldots,b_n $ of $J$ 
 with $h( b_i,b_i ) \in \{1,p\}$ for all $i$.}
 
 \smallskip

We denote by $U$ the $\BZ_{p^2}$-submodule of $\mathbb V$ generated by $j_1,\ldots,j_n$, so that $U\subseteq J \subseteq \mathbb V$. Both $U$ and $J$ are free $\BZ_{p^2}$-modules of rank $n$. Let $U^{\vee}$ (resp. $J^{\vee}$) be the set of $j\in \mathbb V$ with $h(j,c)\in \BZ_{p^2}$ for all $c\in U$ (resp. all $c\in J$). 
It follows that $pU^{\vee}=U$. Let $c_1,\ldots,c_n$ be an orthogonal basis of $J$, and denote by  $\alpha_i$ the  valuation of $h(c_i,c_i)$. Assume now that $\alpha_i\geq 2$ for some $i$. Then it follows that $p^{-1}c_i\in pJ^{\vee}\subseteq pU^{\vee}=U$. Hence $c_i\in pU\subseteq pJ$. But an element of $pJ$ cannot be a member  of a basis of $J$. Hence all $c_i$ have valuation $0$ or $1$. Thus the claim follows.

The number of $b_i$ which have valuation $0$ is positive because there is an element of valuation $0$ in $J$ (e.g. $j_0$). Hence by the induction hypothesis $\CZ(J)_p$ is regular.
We know that the dimension of $\CZ_p$ is the dimension of the supersingular locus of $\mathcal{N}_p$. Therefore the dimension of $\CZ(J)_p=\CZ_p \cap \CZ(j_0)_p$ is smaller than the dimension of $\CZ_p$. Since $\CZ(J)_p$ is regular, it follows that $\CZ_p$ is regular at $x$. 
\end{proof}

\begin{remark} Consider the isogeny $\alpha: \overline{\BY}^n\to X$ defined by $(j_1,\ldots,j_n)$. The kernel of $\alpha$ is a finite flat group scheme $G$ of rank $p^m$, of type $(p, p,\ldots, p)$ and equipped with an action of $\BF_{p^2}$. As Zink pointed out, if $m=1$, such a group scheme can only exist over a base $Y$ with $p\cdot\CO_Y= 0$. (He uses Oort-Tate theory to show this.) We do not know whether Theorem \ref{specialdQ} can be seen from this angle in the general case. 

\end{remark}

\section{Proof of Theorem \ref{idealk}}\label{proofidealk}

Choose a
$W$-basis of $B$ as follows. Choose $e_0, e_1 \ldots, e_c \in B$ such that  $e_1,\ldots,e_c$ project to vectors in $B^*/pB$ and $e_0$ is in $B\setminus B^*$, and such that the images of these vectors in $B/L$ span the Jordan block relative to the eigenvalue $\lambda$ of $\bar{g}$ in $U/U^\perp$. Next let $l$ be the minimal integer $\geq 0$ such that $g^lu\in pB$. For $l>i \geq 0$ denote by $e_{c+1+i}$ the element $g^iu\in L$. Finally, we complete this to a basis by lifting vectors which project to Jordan blocks other than $\lambda$. These last vectors we call $e_{c+l+1},\ldots,e_{n-1}$. We therefore obtain the following identities modulo $L$, 
\begin{equation}\label{basis1}
ge_0\equiv\lambda e_0+e_1, ge_1\equiv \lambda e_1+e_2,\ldots,ge_{c-1}\equiv\lambda e_{c-1}+e_c, ge_c\equiv\lambda e_c .  
\end{equation}
If $m>c+l$ and $e_m,\ldots,e_{m'}$ give rise to a Jordan block of $g$ in $B/L$ to an eigenvalue $\mu$, then 
\begin{equation}\label{basis2}
ge_m\equiv\mu e_m+e_{m+1},\ldots,ge_{m'-1}\equiv\mu e_{m'-1}+e_{m'}, ge_{m'}\equiv\mu e_{m'} . 
\end{equation}
By perhaps changing the $e_i$ by adding a suitable element of $L$, we may (and will) assume that these congruences also hold modulo $pB$.

The vectors $e_i$ form a $W$-basis of $M_0=B$, where $M$ is the Dieudonn\'e module of $X$, the $p$-divisible group belonging to $B$. Let $f_0,\ldots,f_{n-1}$ be a basis of $M_1$ such that $\langle e_i,f_j\rangle=\delta_{ij}$.
Denote by  $T$ the $W$-span of $e_0,f_1,f_2,\ldots,f_{n-1}$ and by $L'$ the $W$-span of $f_0,e_1,e_2,\ldots,e_{n-1}$. (We only write $L'$ instead of the usual notion $L$ since the letter $L$ is already used.) Then
\[
M=L'\oplus T, \ \ \ VM=L'\oplus pT.
\]
Let $h_1=e_0,h_2=f_1,\ldots,h_n=f_{n-1},h_{n+1}=f_0,h_{n+2}=e_1,\ldots,h_{2n}=e_{n-1}.$ 
Define the matrix $(\alpha_{ij})$ by 
\[
Fh_j=\sum_i \alpha_{ij}h_i \text{ for } j=1,\ldots,n,  \\
 \]
 \[
 V^{-1}h_j=\sum_i \alpha_{ij}h_i \text{ for } j=n+1,\ldots,2n.
\]
It follows (see \cite{Zi}, p. 48) that the universal deformation of $X$ over $\BF[\![ t_{11},\ldots,t_{nn}]\!]$ corresponds to the display $(L'\oplus T)\otimes W(\BF[\![t_{11},\ldots,t_{nn} ]\!]) $ with matrix $(\alpha_{ij})^{\text{univ}}$ (wrt. the basis $h_1,\ldots,h_{2n}$ and with entries in $W(\BF[\![t_{11},\ldots,t_{nn} ]\!])$ given by 
$$
(\alpha_{ij})^{\text{univ}}=
\begin{pmatrix}
1&&&[t_{11}]&\hdots&[t_{1n}] \\
&\ddots&&\vdots&\ddots&\vdots \\
&&1&[t_{n1}]&\hdots&[t_{nn}] \\
&&&1&& \\
&&&&\ddots& \\
&&&&&1 \\
\end{pmatrix} \cdot
(\alpha_{ij}).
$$
Here the $[t_{i j}]$ denote the Teichm\"uller representatives of the $t_{i j}$. 
  Now let $A^{'}=W[\![t_{11},\ldots,t_{nn}]\!]$ and let  $R^{'}=\BF[\![t_{11},\ldots,t_{nn}]\!]$. 
  We extend the Frobenius  $\sigma $ on $W$ to $A^{'}$ 
  putting $\sigma(t_{ij})=t_{ij}^p.$ 
 Let $R$ be the completed universal deformation ring (in the special fiber) of $X$ together with the $\BZ_{p^2}$-action and the $p$-principal polarization. Then $R$ is a quotient of $R^{'}$ by an ideal $\mathfrak J$. Using the fact that $(\alpha_{ij})^{\text{univ}}$ has to respect the $\BZ/2$ grading,  it is easy to see that the ideal describing the deformation of the $\BZ_{p^2}$-action is $(t_{11},t_{ij})_{i,j\neq 1}$. Using this,  it is easy to see that 
 $\mathfrak  J=((t_{11},t_{ij})_{i,j\neq 1}, (t_{1i}-t_{i1})_{i\leq n})$. (Compare also \cite{G}, p. 231.) Thus we may identify $R$ with the ring $\BF[\![t_{1},\ldots,t_{n-1}]\!]$, where $t_i$ corresponds to the image of $t_{1i+1}$ in  $R^{'}/\mathfrak  J$. We also define 
 $A=W[\![t_{1},\ldots,t_{n-1}]\!]$.

Let $(\beta_{ij})^{\rm univ}$ be the matrix over $A^{'}$ which is obtained from $(\alpha_{ij})^{\rm univ}$ by replacing the $[t_i]$ by $t_i$ and by multiplying the last $n$ rows by $p$. 
 We consider the $A^{'}$ - $R^{'}$ window $(M^{'},M^{'}_1, \Phi^{'})$ given by 
$
M^{'}=M\otimes A^{'},  \ M_1^{'}=VM\otimes A^{'}, \  
\Phi^{'}=(\beta_{ij})^{\text{univ}}\sigma, \  \,
$
where 
the matrix of $\Phi^{'}$ is described in the basis $h_1,\ldots,h_{2n}$. (We consider the $h_i$ as elements in $M^{'}$, and they form a basis of $M^{'}$; similarly $ph_1,\ldots,ph_n,h_{n+1},\ldots,h_{2n}$ form a basis of $M_1^{'}$.) The corresponding display is the universal display described above (easy to see using the procedure described on p.2 of \cite{Zi2}). Hence we call $(M^{'},M^{'}_1, \Phi^{'})$ the universal window.

For an element $f=\sum a_{k_1,\ldots,k_{n-1}}t_1^{k_1}\cdots t_{n-1}^{k_{n-1}}\in R$ we denote by $\tilde f\in A$  the element $\tilde f=\sum \tilde a_{k_1,\ldots,k_{n-1}}t_1^{k_1}\cdots t_{n-1}^{k_{n-1}}$, where $ \tilde a_{k_1,\ldots,k_{n-1}}$ is the Teichm\"uller lift of $a_{k_1,\ldots,k_{n-1}}$. Thus $\tilde f$ is a lift of $f$.

Let $\mathfrak m$ be the maximal ideal of $R$. In the sequel, we call an ideal $J\subseteq \mathfrak m$  of $R$ {\emph{admissible}}
if $R/J$ is 
isomorphic to $ \BF[T]/(T^l)$ for some $l$ with $1\leq l\leq p$. (In particular, $J$ contains $\mathfrak m^p$.)
 
 Let $J$ be admissible. We now construct a frame for $R/J$.
If $l=1$ (i.e. $J=\mathfrak m$) then there is nothing to do since $W$ is a frame for $\BF$. Thus we may assume that $l\geq 2$. Let $\mathfrak m_{J}$ be the maximal ideal of $R/J$.
The map $R\rightarrow R/J$ induces a surjective linear map of $\BF$-vector spaces $$\phi: \mathfrak m/ \mathfrak m ^2\rightarrow  \mathfrak m_{J}/ \mathfrak m_{J}^2.$$
Since $R/J\cong \BF[T]/(T^l)$, the dimension of  $\mathfrak m_{J}/ \mathfrak m_{J}^2$ is $1$.  Let $\bar X_1\in \mathfrak m/ \mathfrak m^2$ be an element that is not in the kernel of $\phi$. We can extend $\bar X_1$ to a basis $\bar X_1,\ldots,\bar X_{n-1}$ of $\mathfrak m/ \mathfrak m^2$ such that $\bar X_2,\ldots,\bar X_{n-1}$ are in the kernel of $\phi$. Let $X_1\in \mathfrak m$ be any lift of $\bar X_1$. We find lifts $X_2,\ldots, X_{n-1}$ of $\bar X_2,\ldots,\bar X_{n-1}$ which are all contained in $J$.  It follows that $J=(X_1^l,X_2,\ldots,X_{n-1})$. 
Let $\tilde X_i\in A$ be the lifts of the $X_i$ as explained above.
Then it follows that $R=\BF[\![X_1,\ldots, X_{n-1}]\!]$ and  $A=W[\![\tilde X_1,\ldots, \tilde X_{n-1}]\!]$. Define now $\tilde J= ((\tilde X_1)^l,\tilde X_2,\ldots,\tilde X_{n-1})$. Then $(A/\tilde J)/(p)=R/J$,  and $A/\tilde J$ is torsion free as an abelian group. 
The endomorphism on $A$ extending the Frobenius on $W$ by sending  $t_i\mapsto t_i^p$ induces an endomorphism on $A/\tilde J$ sending the images of the $t_i$ to $0$ since we assume that $l\leq p$ and hence $(t_1,\ldots,t_{n-1})^p=(\tilde X_1,\ldots,\tilde X_{n-1})^p\subseteq (\tilde X_1^p,\tilde X_2,\ldots,\tilde X_{n-1}) \subseteq \tilde J$. 
Since furthermore the ideal $p\cdot A/\tilde J$ in $A/\tilde J$ is obviously equipped with a pd-structure,  it follows that indeed $A/\tilde J$ is a frame for $R/J$.

 Let $I$ be the ideal of $\mathcal{M}\cap \mathcal{Z}(g)$ in  $R$. 
By Lemma \ref{tangspace}, $R/I\cong \BF[T]/(T^l)$ for some $l\geq 1$.
Thus, the ideal $I+ \mathfrak m^p$ is the smallest admissible ideal $J$ such that $\Spec (R/J) \subseteq \mathcal{M}\cap \mathcal{Z}(g)$.

Let $J\subseteq R$ be an admissible ideal such that $I+ \mathfrak m^p\subseteq J$.

By base change from the universal window $(M^{'},M^{'}_1, \Phi^{'})$ we  obtain a window  $(M^{(J)},M^{(J)}_1, \Phi^{(J)})$ over $R/J$.
We have a $\BZ/2$-grading\footnote{As in the previous section, we now write the grading index as an upper index, to avoid a conflict of notation with Zink's theory.}  $M^{(J)}={M^{(J),0}}\oplus {M^{(J),1}}$ and  $M^{(J)}_1=M^{(J),0}_1\oplus M^{(J),1}_1$. 
We denote by $G^{(J)}$ the matrix of $g$ wrt.  the  basis of $M^{(J)}$ coming from the above basis of $M'$, and by $G_1^{(J)}$ the matrix of $g$ wrt.  the basis of $M_1^{(J)}$ coming from the above basis of $M_1'$. Denote by $\tilde{\Phi}$ the matrix of $\Phi$ wrt. this  basis of $M^{(J)}$. 
Then $G^{(J)}$ and $G_1^{(J)}$ are integral.  Since $g$ commutes with $\Phi$,  we have 
$$G^{(J)}=\tilde{\Phi}\sigma(G^{(J)})\tilde{\Phi}^{-1}=\tilde{\Phi}\sigma(G^{(\mathfrak m)})\tilde{\Phi}^{-1}.$$
Here we use that since $\mathfrak m^p\subseteq J$, we have $\sigma(G^{(J)})=\sigma(G^{(\mathfrak m)})$, where we view $G^{(\mathfrak m)}$ (i.e. the matrix of $g$ over $W$ corresponding to the $\BF$-valued point) as a matrix with entries in $A/\tilde J$.
Similarly, $$G_1^{(J)}=\begin{pmatrix}
p^{-1}1_n\\
&1_n
\end{pmatrix}\tilde{\Phi}\begin{pmatrix}
p^{-1}1_n\\
&1_n
\end{pmatrix}^{-1}
\sigma(G_1^{(J)})\begin{pmatrix}
p^{-1}1_n\\
&1_n
\end{pmatrix}\tilde{\Phi}^{-1}\begin{pmatrix}
p^{-1}1_n\\
&1_n
\end{pmatrix}^{-1}.
$$
Again we observe that 
 $\sigma(G_1^{(J)})=\sigma(G_1^{(\mathfrak m)})$.
Let 
$$ 
\tau=\begin{pmatrix}
0&t_1&\hdots&t_{n-1}\\
t_1&0&\hdots&0\\
\vdots&\vdots&&\vdots \\
t_{n-1}&0&\hdots&0 \\
\end{pmatrix} .$$  Then 
\begin{equation}\label{Phi}
\tilde{\Phi}=\begin{pmatrix}
1_n&\tau\\
0&1_n
\end{pmatrix}
(\alpha_{ij})\begin{pmatrix}
1_n\\
&p\cdot1_n
\end{pmatrix}.
\end{equation}

If $S\neq 0$ is a quotient of $R$ by an admissible ideal 
and if $(M^S,M_1^S,\Phi^S)$ denotes the corresponding window obtained by base change from the universal one, let $g^S: M^S\otimes \BQ \rightarrow M^S \otimes \BQ$ be the map which is  induced by the map  $g^{(\mathfrak m^p)}:M^{(\mathfrak m^p)}\otimes \BQ\rightarrow M^{(\mathfrak m^p)}\otimes \BQ$ which in turn lifts $g$ and commutes with $\Phi^{(\mathfrak m^p)}$. 
(Here $(M^{(\mathfrak m^p)}, M_1^{(\mathfrak m^p)}, \Phi^{(\mathfrak m^p)})$ is the window over $R/\mathfrak m^p$ obtained by base change from the universal one, where we use the obvious frame $A/(t_1,...,t_{n-1})^p$ for $R/\mathfrak m^p$.)
The map $g^{(\mathfrak m^p)}$ is given by the matrix $G^{(\mathfrak m^p)}$ of $g^{(\mathfrak m^p)}$ with respect to the above basis of $M^{(\mathfrak m^p)}\otimes \BQ$, i.e.  $G^{(\mathfrak m^p)}=\tilde{\Phi}\sigma(G^{(\mathfrak m^p)})\tilde{\Phi}^{-1}=\tilde{\Phi}\sigma(G^{(\mathfrak m)})\tilde{\Phi}^{-1}$, where again by abuse of notation we write $G^{(\mathfrak m)}$ for an arbitrary lift of $G^{(\mathfrak m)}$ over $A/(t_1,...,t_{n-1})^p$ and $\tilde{\Phi}$ is the matrix of  $\Phi^{(\mathfrak m^p)}$ wrt. the above basis of $M^{(\mathfrak m^p)}\otimes \BQ$. 
We claim that $g$ lifts over $S$ if and only if $g^S$ maps $M^{S,0}$ into $M^{S,0}$ and $M_1^{S,0}$ into $M_1^{S,0}$. It is obvious that these conditions are necessary. Suppose they are fulfilled.

Since $g$ is unitary, $\langle x,y\rangle=\langle gx,gy\rangle$ for all $x,y\in M$. In other words, for the adjoint $g^{\dagger}$ of $g$ we have $g^{\dagger}=g^{-1}$. It is obvious (from the above formulas for the matrices of the lifts of $g$) that  the map $p^{2}g$ lifts to a map $\tilde{g}_1$ over $R/\mathfrak m^p$.  By rigidity $\tilde{g}_1^{\dagger}=p^{4}\tilde{g}^{-1}_1$. Let  $\langle, \rangle_{p}$ be the alternating form on $M^{(\mathfrak m^p)}\otimes_{\BZ} \BQ$. Then it follows that $\langle p^{-2}\tilde{g}x,y \rangle_{p}=\langle x,(p^{-2}\tilde{g})^{-1}y \rangle_{p}$. Suppose now that  $g^S$ maps $M^{S,0}$ into $M^{S,0}$ and $M_1^{S,0}$ into $M_1^{S,0}$. Denoting  by $\langle, \rangle_S$ the alternating form on $M^S$,  this means that $\langle g^Se_i, f_j \rangle_S$ is integral for all $i,j $ and $\langle g^Se_i, f_1 \rangle_S$ is an integral multiple of $p$ for all $i$. This implies that $\langle e_i, (g^S)^{-1}f_j \rangle_S$ is integral for all $i,j $ and $\langle e_i, (g^S)^{-1}f_1 \rangle_S$ is an integral multiple of $p$ for all $i$. Since the determinant of $g$ (restricted to an endomorphism of $M^1\otimes_{\BZ}\BQ$) is a unit in $W$, also the determinant of  $g^S$ (restricted to an endomorphism of ${M^{S,1}}\otimes_{\BZ}\BQ$) is a unit in $S$. Hence it follows that also  $\langle e_i, (g^S)f_j \rangle_S$ integral for all $i,j $ and $\langle e_i, (g^S)f_1 \rangle_S$ is an integral multiple of $p$ for all $i$. This shows that $g^S$ maps $M^{S,1}$ into $M^{S,1}$ and $M_1^{S,1}$ into $M_1^{S,1}$, confirming the claim.

\noindent Let $J\subseteq R$ again be an admissible ideal such that $I+ \mathfrak m^p\subseteq J$. We compute: 
\begin{equation*}
\begin{aligned}
 G_1^{(J)}&=\begin{pmatrix}
p^{-1}1_n\\
&1_n
\end{pmatrix}\tilde{\Phi}\begin{pmatrix}
p^{-1}1_n\\
&1_n
\end{pmatrix}^{-1}
\sigma\big(G_1^{(\mathfrak m)}\big)\begin{pmatrix}
p^{-1}1_n\\
&1_n
\end{pmatrix}\tilde{\Phi}^{-1}\begin{pmatrix}
p^{-1}1_n\\
&1_n
\end{pmatrix}^{-1}\\
&=\begin{pmatrix}
p^{-1}1_n\\
&1_n
\end{pmatrix}\tilde{\Phi}
\sigma\big(G^{(\mathfrak m)}\big)\tilde{\Phi}^{-1}\begin{pmatrix}
p^{-1}1_n\\
&1_n
\end{pmatrix}^{-1}.
\end{aligned}
\end{equation*}
Using that $G^{(\mathfrak m)}=(\alpha_{ij})\begin{pmatrix}
1_n& 0\\
0&p1_n
\end{pmatrix}\sigma\big(G^{(\mathfrak m)}\big)\begin{pmatrix}
1_n& 0\\
0&p1_n
\end{pmatrix}^{-1}(\alpha_{ij})^{-1} $ and the equation \eqref{Phi}, we obtain 
$$
G_1^{(J)}=
\begin{pmatrix}
p^{-1}1_n\\
&1_n
\end{pmatrix}
\begin{pmatrix}
1_n&\tau\\
0&1_n
\end{pmatrix}
G^{(\mathfrak m)}\begin{pmatrix}
1_n&-\tau\\
0&1_n
\end{pmatrix}\begin{pmatrix}
p^{-1}1_n\\
&1_n
\end{pmatrix}^{-1}.
$$
This matrix respects the $\BZ/2$-grading and it is integral since 
$I+\mathfrak m^p\subseteq J$.
We consider the block matrix of $g=g^{R/\mathfrak m}$ 
according to the $\BZ/2$-grading of $M$, and denote by $H$ the `left upper block' describing the endomorphism of $M^0$ induced by $g$ wrt  the basis $e_0,\ldots,e_{n-1}$. 
Let $H^{(J)}$ the upper left block of the matrix obtained from $G^{(J)}$ by base change to the basis $e_0,\ldots,e_{n-1}, f_0,\ldots,f_{n-1}$. Similarly, let $H_1^{(J)}$ be the upper left block of the matrix obtained from $G_1^{(J)}$ by base change to the basis $pe_0,e_1\ldots,e_{n-1}, f_0,pf_1,\ldots,pf_{n-1}$.
Then $H^{(J)}_1$ is given by

$$ H^{(J)}_1=
\begin{pmatrix}
p^{-1}\\
&1_{n-1}
\end{pmatrix}
\begin{pmatrix}
1&t_1&\ldots&t_{n-1}\\
&1\\
&&\ddots \\
&&&1
\end{pmatrix}
H
\begin{pmatrix}
1&-t_1&\ldots&-t_{n-1}\\
&1\\
&&\ddots \\
&&&1
\end{pmatrix}
\begin{pmatrix}
p^{-1}\\
&1_{n-1}
\end{pmatrix}^{-1}.
$$
Note that we only want to consider deformations which factor through $\delta(\mathcal{M})$. In terms of the parameters $t_1,\ldots,t_{n-1}$, this condition just says  that $t_{c+1}=0$. 

The matrix $H$ is up to multiples of $p$ given by the description of the action of $g$ on the $e_i$ in the beginning of the proof.

We want to find the conditions that the matrix $H^{(J)}_1$ is integral.  For this it is enough to check when the entries in the first  line   are integral (the coefficients of $pe_0$ in the images of the basis vectors). Using the equations \eqref{basis1} (and calculating modulo integral elements, i.e., modulo elements of $R$), the first $c+l$ entries are  
\begin{equation*}
\begin{aligned}
\lambda+t_1, (-t_1^2+t_2)/p, (-t_1t_2+t_3)/p,\ldots,(-t_1t_{c-1}+t_c)/p,(-t_1t_c)/p, \\
\big(-(\lambda+t_1)t_{c+1}+t_{c+2}\big)/p,\ldots, \big(-(\lambda+t_1)t_{c+l-1}+t_{c+l}\big)/p, \big(-(\lambda+t_1)t_{c+l}\big)/p .
\end{aligned}
\end{equation*}
Using $t_{c+1}=0$, this shows  that in $R/J$
$$t_1t_i=t_{i+1},\,  \forall i\leq c-1,t_1t_c=0, \text{ and } t_{c+2}=\ldots=t_{c+l}=0 .$$
If $e_m,\ldots,e_{m'}$ span another Jordan block (mod $L$) to an eigenvalue $\mu$, then by equations (\ref{basis1}) and (\ref{basis2}), the entries with index between $m$ and $m'$  are (modulo integral elements)
\begin{equation*}
\begin{aligned}
\big(-\lambda t_{m}-t_1 t_{m}+\mu t_{m}+ t_{m+1}\big)/p,\ldots, \big(-\lambda t_{m'-1}-t_1 t_{m'-1}+\mu t_{m'-1}+ t_{m'}\big)/p, t_{m'}(-\lambda+\mu-t_1)/p .\end{aligned}
\end{equation*}
 Since $\lambda\neq \mu$ (modulo $p$), the expression $(-\lambda+\mu-t_1)$ is a unit, hence we obtain $t_{m'}=0$ in $R/J$. Inductively we obtain that $t_m=\ldots=t_{m'-1}=t_{m'}=0$ in $R/J$. Therefore 
$$
\big(t_1t_i-t_{i+1} \text{ for } i\leq c-1, t_1t_c, t_i \ \text{for } i\geq c+1\big)\subseteq J.$$
Note that there do not occur further conditions from the integrality of $$ H^{(J)}=
\begin{pmatrix}
1&t_1&\ldots&t_{n-1}\\
&1\\
&&\ddots \\
&&&1
\end{pmatrix}
H
\begin{pmatrix}
1&-t_1&\ldots&-t_{n-1}\\
&1\\
&&\ddots \\
&&&1
\end{pmatrix}.
$$

We claim that $J_0:=\big(t_1t_i-t_{i+1} \text{ for } i\leq c-1, t_1t_c, t_i \ \text{for } i\geq c+1\big)+\mathfrak m^p=I+m^p$. 
Obviously $J_0$ is admissible. If we choose $S=R/J_0$ then the same calculation as above shows that $g$ lifts over $S$.  Since $t_{c+1}\in J_0$ it follows that $\Spec (S)\subseteq \mathcal{M}\cap \mathcal{Z}(g)$.
It follows that indeed $I+m^p=J_0$.

Since $2c+1\leq n$ and since we assume $n \leq 2p-2$, we have $c+1< p$. Hence the claim of Theorem \ref{idealk} follows from the next lemma. 
 \qed
 \begin{lemma}
 Let $r>s>0$ be integers. Consider the ideal $J$ in $\BF[\![X_1,\ldots,X_n]\!]$, where 
 $$
 J=(X_1^s, X_2,\ldots,X_n) .
 $$
 Let $I$ be an ideal in $\BF[\![X_1,\ldots,X_n]\!]$ such that 
 $$I+\mathfrak m^r=J,
 $$ where $\mathfrak m$ denotes the maximal ideal. Then $I=J$. 
 
 \end{lemma}
 \begin{proof}
We proceed in several steps.

\smallskip

\noindent \emph{Step 1.} Consider the projection $\BF[\![X_1,\ldots,X_{n}]\!]\to\BF[\![X_2,\ldots,X_{n}]\!]$, obtained by dividing out by $(X_1)$. Let $\bar I,$ resp. $ \bar J,$ be  the image of $I$, resp.  $J$, and let $\bar{ \mathfrak m}$ be  the maximal ideal of 
$\BF[\![X_2,\ldots,X_{n}]\!]$. Then $\bar I=\bar J=\bar{\mathfrak m}$ by Nakayama's Lemma. 

\smallskip

\noindent \emph{Step 2.} Let $b\in J$. Then $b$ is congruent modulo $I$ to an element of $\mathfrak m^r$. Writing this latter element as a sum of monomials in $X_1,\ldots, X_n$, and using step 1, we see that $b$ is congruent modulo $I$ to an element in the ideal $(X_1^r)$. Hence it suffices to prove that $X_1^r\in I$. 

\smallskip

\noindent \emph{Step 3.} We claim that in fact  $X_1^s\in I$.  We will show that $X_1^s\in I+\mathfrak m^{kr}$ for all $k$, which will prove the claim. 

We proceed by induction on $k$, the case $k=1$  holding true by hypothesis. Assume that $X_1^s\in I+\mathfrak m^{kr}$. Hence we are assuming that $X_1^s$ is congruent modulo $I$ to an element of $\mathfrak m^{kr}$. Writing this element of $\mathfrak m^{kr}$ as a sum of monomials in $X_1,\ldots,X_n$, we subdivide this sum into 
\begin{itemize}
\item a sum of monomials, where the exponent of $X_1$ is $\geq s$,
\item a sum of monomials, where the exponent of $X_1$ is $<s$.
\end{itemize}
In the first sum, we extract the factor $X_1^s$; since $kr>s$, the remainder lies in $\mathfrak m$. Bringing this first sum to the left hand side, we see that this expression differs from $X_1^s$ by a unit. Hence we may disregard the first sum.

In the second sum, the total degree in $X_2,\ldots,X_n$ of each monomial is strictly larger than $kr-s$, i.e., is at least $ (k-1)r+2$. Now  $X_2,\ldots,X_n$ are congruent modulo $I$ to elements in $\mathfrak m^r$ and  we may replace each  $X_2,\ldots,X_n$ by an element in $\mathfrak m^r$. Then  each summand lies in $\mathfrak m^{((k-1)r+2)r}\subset \mathfrak m^{(k+1)r}$, which concludes the induction step.

 \end{proof}
\vspace{2cm}


%
%
%
%

\end{document}